\documentclass{amsart}

\usepackage{amssymb,color}
\usepackage{amsfonts}
\usepackage{amsmath}
\usepackage{euscript}
\usepackage{enumerate}
\usepackage{graphics}

\usepackage{hyperref}
\usepackage{pdfsync}
\newtheorem{theorem}{Theorem}[section]
\newtheorem{lemma}[theorem]{Lemma}

\newtheorem{prop}[theorem]{Proposition}
\newtheorem{cor}[theorem]{Corollary}

\newtheorem{ques}[theorem]{Question}
\newtheorem{notation}[theorem]{Notation}

\newtheorem*{Theorem1'}{Theorem 1'}

\theoremstyle{definition}
\newtheorem{definition}[theorem]{Definition}
\newtheorem{example}[theorem]{Example}

\theoremstyle{remark}
\newtheorem{remark}[theorem]{Remark}

\numberwithin{equation}{section}

\newcommand \C{{\mathbb C}}
\newcommand \F{{\mathbb F}}
\renewcommand \ker{{\mathrm {ker}}}

\newcommand \rk{{\mathrm {rk}}}
\newcommand \GL{{\mathrm {GL}}}
\newcommand \gl{{\mathfrak {gl}}}

\newcommand \g{{\mathfrak {g}}}

\newcommand \n{{\mathfrak {n}}}
\newcommand \h{{\mathfrak {h}}}

\newcommand \la{{\lambda}}

\newcommand \al{{\alpha}}

\renewcommand \sl{{\mathfrak {sl}}}

\newcommand \B{{\mathcal B}}
\newcommand \ad{{\mathrm {ad}}}
\newcommand \N{{\mathcal N}}
\renewcommand \L{{\mathcal L}}
\newcommand \p{{\mathfrak {p}}}

\begin{document}

\title[Nilpotency degree of the nilradical of a solvable Lie algebra]{Nilpotency degree of the nilradical of a solvable Lie algebra on two generators}

\author{Leandro Cagliero}
\address{FaMAF-CIEM (CONICET), Universidad Nacional de C\'ordoba,
Medina Allende s/n, Ciudad Universitaria, 5000 C\'ordoba, Rep\'ublica
Argentina.}
\email{cagliero@famaf.unc.edu.ar}

\author{Fernando Levstein}
\address{FaMAF-CIEM (CONICET), Universidad Nacional de C\'ordoba,
Medina Allende s/n, Ciudad Universitaria, 5000 C\'ordoba, Rep\'ublica Argentina.}
\email{levstein@famaf.unc.edu.ar}

\author{Fernando Szechtman}
\address{Department of Mathematics and Statistics, Univeristy of Regina, Canada}
\email{fernando.szechtman@gmail.com}
\thanks{This research was partially supported by an NSERC grant, CONICET
PIP 112-2013-01-00511, 
PIP 112-2012-01-00501, 
MinCyT C\'ordoba, 
FONCYT Pict2013 1391, 
SeCyT-UNC
33620180100983CB}





\begin{abstract} Given a sequence $\vec d=(d_1,\dots,d_k)$ of natural numbers, we
consider the Lie subalgebra $\h$ of $\gl(d,\F)$, where $d=d_1+\cdots +d_k$ and $\F$
is a field of characteristic 0, generated by two block upper triangular
matrices $D$ and $E$ partitioned according to $\vec d$, and study the
problem of computing the nilpotency degree $m$ of the nilradical $\n$ of $\h$.
We obtain a complete answer when $D$ and $E$ belong to a certain family of
matrices that arises naturally when attempting to classify the indecomposable
modules of certain solvable Lie algebras.

Our determination of $m$ depends in an essential manner on the symmetry of $E$
with respect to an outer automorphism of $\sl(d)$. The proof that $m$ depends solely on this symmetry is long and delicate.

As a direct application of our investigations on $\h$ and $\n$ we give a full
classification of all uniserial modules of an extension of the free $\ell$-step
nilpotent Lie algebra on $n$ generators when $\F$ is algebraically closed.
\end{abstract}

\maketitle

\section{Introduction}
\label{intro}

Uniserial representations play the role of building blocks in the in the representation theory of non-semisimple associative or Lie algebras.
A celebrated result of T. Nakayama \cite{Na} states that every finitely generated module over a serial ring is a direct sum of uniserial modules. This triggered
a systematic study of uniserial representations of associative algebras.
In the 1970's, M. Auslander posed the problem of determining which Artin algebras of infinite representation type have only a finite number of non-isomorphic
uniserial modules, and significant progress in this problem was  attained
by B. Huisgen-Zimmermann  \cite{H-Z2}, who had previously introduced an affine variety
to describe the isomorphism classes of
uniserial modules of a given split basic algebra \cite{H-Z}.
W. Burgess and B. Huisgen-Zimmermann \cite{BH-Z}
used uniserial modules to approximate  classes of less accessible modules, 
and Z. Nazemian, A. Ghorbani and  M. Behboodi \cite{NGB}
introduced the uniserial dimension to measure how
far a module deviates from being uniserial.
C. Conley dealt with the classification of uniserial extensions of tensor field modules
in \cite{Co}

Even though the uniserial representation theory is  well developed for
associative algebras, very little is known for Lie algebras.
Recently, the literature shows an increasing interest in understanding certain type of
indecomposable  representations of (non-semisimple)
Lie algebras,
see for example \cite{RdG, DP, CMS, J, CS_JofAlg, C, CGS1, CGS2, DdG, CPS, DR, DP}.

\

In this paper we consider the Lie algebras
\[
\g_{n,\lambda}=\langle x\rangle\ltimes \L(n)\quad\text{ and }\quad
\g_{n,\lambda,\ell}=\langle x\rangle\ltimes \N_\ell(n)
\]
where $\L(n)$ is the free Lie algebra on $n$ generators,
$\N_\ell(n)=\L(n)/\L(n)^\ell$ is the corresponding
free $\ell$-step nilpotent Lie algebra,
and $x$ acts as the derivation of $\L(n)$ that restricted to
$\F^n\subset\L(n)$ acts by a single Jordan block $J_n(\la)$
(here $J_p(\al)$ (resp. $J^p(\al)$) denotes the lower (resp. upper)
triangular Jordan block of size $p$ and eigenvalue $\al$).
In this work we give a complete classification, up to isomorphism,
of all finite dimensional uniserial
representations of  $\g_{n,\lambda}$ and $\g_{n,\lambda,\ell}$, $\lambda\ne0$,
when $\F$ is an algebraically closed field of characteristic zero.

In order to describe the classification we need to
introduce the following finite dimensional Lie algebras of matrices.
Let $M_{a\times b}$ denote the vector
space of matrices with entries in $\F$ of size $a\times b$.
Given
\begin{itemize}
\item[{\tiny $\bullet$}] scalars $\al,\la\in \F$,
    \item[{\tiny $\bullet$}] a sequence $(d_1,\dots,d_{\ell+1})$ of $\ell+1$ integers, $\ell\ge1$,
    \item[{\tiny $\bullet$}] a sequence $S=(S(1),\dots,S({\ell}))$ with
    $S(i)\in M_{d_{i}\times d_{i+1}}$ and $S(i)_{d_i,1}\neq 0$,
\end{itemize}
\smallskip
let $D(\alpha, \lambda),E(S)\in\gl(d)$,
$d=\sum d_i$, given in block form by
$$
D(\alpha, \lambda)=J^{d_1}(\al)\oplus J^{d_2}(\al-\la)\oplus\cdots\oplus J^{d_{\ell+1}}(\al-\ell\la),
$$
and
$$
E(S)=\left(
               \begin{array}{ccccc}
                 0 & S(1) & 0 & \dots & 0 \\
                 0 & 0 & S(2) & \dots & 0\\
                 \vdots & \vdots & \ddots & \ddots & \vdots\\
                 0 & 0 & \dots & \ddots & S(\ell) \\
                 0 & 0 & \dots & \dots & 0 \\
               \end{array}
             \right).
$$
Let $\h(\alpha, \lambda,S)$ be the Lie subalgebra of $\gl(d)$ generated by
$D(\alpha, \lambda)$ and $E(S)$.

It is not difficult to see that when
\begin{equation*}\label{eq.intro1}
\max\{d_i+d_{i+1}:1\le i\le \ell\}=n+1,
\end{equation*}
there is a finite dimensional uniserial representation $\pi:\g_{n,\lambda}\to M_{d\times d}$ with image $\h(\alpha, \lambda,S)$
such that $\pi(v)\ne0$ for all $0\neq v\in\F^n\subset\g_{n,\lambda}$.
Theorem \ref{thm.main1} shows that, the restriction of this correspondence to
normalized sequences $S$ (see Definition \ref{def.Normalized})
establishes a bijection with the isomorphism classes of finite dimensional uniserial
representations of  $\g_{n,\lambda}$ such that $\ker\pi\cap \F^n=(0)$ (it suffices to classify these).

For the classification of all finite dimensional uniserial
$\g_{n,\lambda,\ell}$-modules it is necessary to determine the
nilpotency degree of the nilradical $\n$ of $\h(\alpha, \lambda,S)$.
One would expect that, generically, this degree should be $\ell$;
but, even though the situation is clear in the associative context, 
the cancellations appearing in the Lie algebra case
make it very difficult
to   determine exactly when
the nilpotency degree of $\n$ is less than $\ell$.
The answer turned out to be much more interesting than we had
originally expected:
apart from the trivial case $d_i=1$ for all $i$,
the nilpotency degree of $\n$, for normalized $S$, is less that $\ell$
if and only if
\begin{equation}\label{eq.cond_phi}
\begin{split}
 \text{\tiny $\bullet$}\quad  & \text{$\ell$ even and $d_i=d_{\ell+2-i}$
 for  $1\le i\le \ell+1$,} \\
  \text{\tiny $\bullet$}\quad  & \text{$d_1=d_{\ell+1}=1$ and  $d_{\frac{\ell}2+1}$ is odd, } \\
  \text{\tiny $\bullet$}\quad  & \text{$S$ is $\phi$-invariant, }
\end{split}
\end{equation}
where
$\phi:\gl(d)\to\gl(d)$  is the automorphism defined in \eqref{eq.def_phi}, namely a conjugate of minus the transpose map (see Theorem \ref{thm.main2}).
In this case, the degree is $\ell-1$.

\medskip

It is not difficult to  show that condition \eqref{eq.cond_phi}
is sufficient for the nilpotency degree of $\n$ to be less than $\ell$ (see Proposition \ref{ida}).
On the other hand, the proof that condition \eqref{eq.cond_phi}
is also necessary consists of two
quite involved steps.

In the first one, we consider the special case  of  a canonical $S$
(see Definition \ref{def.Normalized}).
To deal with this case,  we make use of
the natural block decomposition of $D(\alpha, \lambda)$ and $E(S)$
and we develop in \S\ref{subsec.canonical} a
sophisticated induction process in which we need to keep track
of the minimal non-zero ranks of each block of $X$, for $X\in\h(\alpha, \lambda,S)$.
The answer to our problem in this case is given by Theorem \ref{thm.main}.

The second step is worked out
in \S\ref{subsec.arbitrary_S}.
Here we consider an arbitrary $S$, which might be viewed as a deformation of the canonical $S$, where the hypothesis
of $\n$ having degree less than $\ell$ implies
that the variables of $S$
satisfy a (huge) system of polynomial equations
(this resembles the affine variety defined in  \cite{H-Z2}).
A crucial step here was to find a normalization of $S$ that surprisingly linearizes  a key part of the polynomial system.
We point out that in \cite{CGS1} a different normalization was used that did not linearize the system as done here.
Finally, in Proposition \ref{eluci}, a subtle Gaussian elimination process is performed to the linear system in order to prove that a normalized $S$
must be $\phi$-invariant and the
main result is Theorem \ref{thm.main2}.
In \S\ref{subsec.arbitrary_S} we also include some examples
showing the contrast between characteristic $0$ and $p$
and pose some open questions.

We close \S\ref{sec.nilp_degree} by determining  which
Lie algebras $\h(\alpha, \lambda,S)$ are free nilpotent
(see Theorem \ref{thm.free}).

We are confident that our results about the
nilpotency degree of  $\n$ will
find applications elsewhere.

\medskip

In  \S\ref{sec.uniserials}  we use the nilpotency degree of
$\n$ to classify all uniserial
representations of $\g_{n,\lambda}$ and $\g_{n,\lambda,\ell}$.
In \S\ref{subsec.R} we introduce a family of normalized
representations
$R_{\vec d,\al,S}$ of  $\g_{n,\lambda}$ and we show in
\S\ref{sec:uniserials_free} which of these yield
the uniserial representations of $\g_{n,\lambda}$.
The precise statement is given in Theorem \ref{thm.main1},
whose proof requires  an adaptation to our context of some machinery developed in previous papers (see \cite{CGS1} and \cite{CPS}).

The classification of all uniserials $\g_{n,\lambda,\ell}$
is realized in \S\ref{subsec.rel_faithful} and
requires the determination of the representations
$R_{\vec d,\al,S}$  of  $\g_{n,\lambda}$ which ``pass to the quotient''.
It is here where the nilpotency degree of $\n$ is needed.
We also need the concept of relatively faithful representation
of $\g_{n,\lambda,\ell}$ to avoid representations of
 $\g_{n,\lambda,\ell}$ that are trivial extensions of a
representation of
 $\g_{n',\lambda,\ell'}$ with either $n'<n$ or $\ell'<\ell$.
 The main result is Theorem \ref{thm.main3}.

\section{Preliminaries and notation}\label{sec.p&n}

We assume throughout the paper that $\F$ is a field of characteristic 0.

\subsection{Gradings in \texorpdfstring{$\gl(d)$}{gl(d)} and the outer  automorphism}\label{subsec.gradings}
If $\vec{d}=(d_1,\dots,d_{k})$ is a sequence of $k$
positive integers, we define
$|\vec d|=|\vec d|_1=d_1+\cdots+d_{k}$.
A sequence $\vec d$ provides $\gl(d)$, $d=|\vec d|$, with a block structure and we define
\[
p_{i,j}: \gl(d)\to M_{d_i\times d_j}
\]
to be the projection onto the $(i,j)$-block.

We consider two `diagonal' gradings in $\gl(d)$:
one associated to the actual diagonals of $\gl(d)$,
that is
\begin{equation}\label{eq.grad1}
\mathcal{D}_t=\{A\in \gl(d):A_{ij}=0\text{ if $j-i\ne t$}\};
\end{equation}
and the other one associated to the block-diagonals of $\gl(d)$, that is
\begin{equation}\label{eq.grad2}
\bar{\mathcal{D}}_t=\{A\in \gl(d):p_{ij}(A)=0\text{ if $j-i\ne t$}\}.
\end{equation}
We call the degrees \eqref{eq.grad1} and \eqref{eq.grad2} \emph{diagonal-degree}
and \emph{block-degree} respectively.
The proof of the following proposition is straightforward.
\begin{prop}\label{prop.degrees}
If $A \in  \mathcal{D}_t$ and, for $i< j$,
$\big(p_{i,j}(A)\big)_{r,s}\ne 0$,
(with $1\le r\le d_i$ and  $1\le s\le d_j$) then
\[
t=d_i+\dots + d_{j-1}+(s-r).
\]
In particular, if either
\[
d_{j+1}-1<d_i-d_j+s-r\quad\text{ or }\quad d_i-d_j+s-r<1-d_{i+1}
\]
then $p_{i+1,j+1}(A)=0$.
Similarly, if either
\[
d_{i-1}-1<s-r\quad\text{ or }\quad s-r<1-d_{j-1}
\]
then $p_{i-1,j-1}(A)=0$.
\end{prop}

Observe that the map $\phi:\gl(d)\to\gl(d)$ given by
\begin{equation}\label{eq.def_phi}
\phi(A)_{i,j}=(-1)^{i-j+1}A_{d+1-j,d+1-i}
\end{equation}
is the automorphism of $\gl(d)$ defined by $\phi(A)=-KA^TK^{-1}$,
where the $A^T$ denotes the transpose of $A$ and $K=(a_{i,j})\in\gl(d)$
is the antidiagonal matrix, satisfying $a_{i,d+1-i}=(-1)^{i+1}$  and $a_{i,j}=0$ if $i+j\ne d+1$.
It is clear that $\phi(\mathcal{D}_t)\subset \mathcal{D}_t$.

\subsection{The Lie algebra \texorpdfstring{$\h(\alpha, \lambda,S)$}{h(a,l,S)}} 

Recall that given an integer $p\geq 1$ and $\al\in \F$,
we write $J_p(\al)$ (resp. $J^p(\al)$) for the lower (resp. upper) triangular Jordan block of size $p$ and eigenvalue $\al$.

Given a 5-tuple $(\al,\la, k, \vec d,S)$, where
\begin{enumerate}[{\tiny $\bullet$}]
\item $\al,\la\in \F$ are scalars,
    \item $\vec d =(d_1,\dots,d_{k})$ is a sequence of $k$ positive integers, $k \geq 2$,
    \item $S=(S(1),\dots,S({k-1}))$ is a sequence of $k-1$ matrices satisfying
    \begin{equation}\label{eq.Condition_S}
    S(i)\in M_{d_{i}\times d_{i+1}}\text{ and }S(i)_{d_i,1}\neq 0\text{ for all }i;
    \end{equation}
\end{enumerate}
we
consider the matrices $D(\alpha, \lambda),E(S)\in\gl(d)$,
$d=|\vec d|$, given in block form by
$$
D(\alpha, \lambda)=J^{d_1}(\al)\oplus J^{d_2}(\al-\la)\oplus\cdots\oplus J^{d_{k}}(\al-(k-1)\la),
$$
and
$$
E(S)=\left(
               \begin{array}{ccccc}
                 0 & S(1) & 0 & \dots & 0 \\
                 0 & 0 & S(2) & \dots & 0\\
                 \vdots & \vdots & \ddots & \ddots & \vdots\\
                 0 & 0 & \dots & \ddots & S(k-1) \\
                 0 & 0 & \dots & \dots & 0 \\
               \end{array}
             \right).
$$
Let $\h(\alpha, \lambda,S)$ be the Lie subalgebra of $\gl(d)$ generated by
$D(\alpha, \lambda)$ and $E(S)$.

\begin{definition}\label{def.Normalized}
Given $\vec d =(d_1,\dots,d_{k})$, let $C=(C(1),\dots,C({k-1}))$ with
\[
C(i)=\left(
               \begin{array}{cccc}
                 0 &  0 & \dots & 0 \\
                 \vdots & \vdots  &  & \vdots  \\
                 0 & 0 & \dots & 0\\
                 1 & 0 & \dots  & 0
               \end{array}
             \right)\in M_{d_{i}\times d_{i+1}}.
\]
We call the sequence $C$ the
\emph{canonical} sequence.
Also, given a sequence  $S=(S(1),\dots,S({k-1}))$ as in \eqref{eq.Condition_S},
we say that $S$ is \emph{normalized} if all the following conditions are satisfied:
\begin{enumerate}[(1)]
    \item $S(i)_{d_i,1}=1$ for all $1\le i\le k-1$;
    \item $S(i)_{d_{i},j}=S(i+1)_{d_{i+1}+1-j,1}$
    for $1\le j\le d_{i+1}$ and $1\le i\le k-1$;
    \item $S(1)_{j,1}=0$ for $1\le j<d_1$, and $S(k-1)_{d_{k-1},j}=0$ for $1< j\le d_{k}$.
\end{enumerate}
We say that $S$ is \emph{weakly normalized} if conditions (1) and (2)
are satisfied (this last concept will be used only in  \S\ref{subsec.arbitrary_S}).
\end{definition}
\begin{example} It is clear that the canonical sequence $C$ is normalized. Also, if $\vec d =(3, 5,3,4)$ and $S=(S(1),S(2),S(3))$ is a normalizad sequence,
then $E(S)$ looks as follows (the $*$ might be any scalar):
\[
E(S)= \tiny
\left(
\begin{array}{ccc|ccccc|ccc|cccc}
0&0&0& 0 & * & * & * & * &0&0&0&0&0&0&0\\[.5mm]
0&0&0& 0 & * & * & * & * &0&0&0&0&0&0&0\\[.5mm]
0&0&0& 1 & a_{2} & a_{3} & a_{4} &a_{5} &0&0&0&0&0&0&0 \\[.5mm] \hline
0&0&0&0&0&0&0&0& a_{5} & * & * &0&0&0&0 \\[.5mm]
0&0&0&0&0&0&0&0& a_{4} & * & * &0&0&0&0 \\[.5mm]
0&0&0&0&0&0&0&0& a_{3} & * & * &0&0&0&0\\[.5mm]
0&0&0&0&0&0&0&0& a_{2} & * & * &0&0&0&0\\[.5mm]
0&0&0&0&0&0&0&0& 1 &  b_{2} & b_{3} &0&0&0&0 \\[.5mm] \hline
0&0&0&0&0&0&0&0&0&0&0& b_{3} & * & * & *  \\[.5mm]
0&0&0&0&0&0&0&0&0&0&0& b_{2} & * & * & *  \\[.5mm]
0&0&0&0&0&0&0&0&0&0&0&1&0&0&0\\[.5mm] \hline
0&0&0&0&0&0&0&0&0&0&0&0&0&0&0\\[.5mm]
0&0&0&0&0&0&0&0&0&0&0&0&0&0&0\\[.5mm]
0&0&0&0&0&0&0&0&0&0&0&0&0&0&0\\[.5mm]
0&0&0&0&0&0&0&0&0&0&0&0&0&0&0
\end{array} \right).
\]
\end{example}

The following proposition is not difficult to prove.

\begin{prop}\label{prop.normalized_h}
Let $\vec d =(d_1,\dots,d_{k})$ and let $G(\vec d )$
be the subgroup of  $GL(d)$,
$d=|\vec d|$, consisting of invertible matrices
$P=P_1\oplus\cdots\oplus P_{k}\in  GL(d)$, with
$P_i$ a polynomial (with non-zero constant term) in $J^{d_i}(0)$.
Given a sequence  $S=(S(1),\dots,S({k-1}))$ as in \eqref{eq.Condition_S},
there is an unique invertible matrix
$P\in G(\vec d)$ such that $PE(S)P^{-1}$ is equal to
$E(S')$ for a normalized sequence $S'$.
\end{prop}

\subsection{The Lie subalgebra \texorpdfstring{$\n(S)$}{n(S)} and the structure of  \texorpdfstring{$\h(\alpha, \lambda,S)$}{h(a,l,S)}}

Set
\[
E^{(l)}(S)=(\ad\, D(0,0))^{l} E(S),\qquad\text{for $l\ge 0$}.
\]
Since $\text{char}\,\F=0$, a straightforward computation (or the representation theory of $\sl(2)$) shows that
the set $\{E^{(l)}(S)\}_{l=0}^\rho$,
with $\rho=\max\{d_i+d_{i+1}-2:i=1,\dots,k-1\}$, is linearly independent.
Let $\n(S)$ be the Lie algebra generated by $\{E^{(l)}(S)\}_{l=0}^\rho$,
that is
\begin{equation}\label{eq.generators}
\n(S)=\text{span}_{\F}[[[E^{(l_1)}(S),E^{(l_2)}(S)],E^{(l_3)}(S)],\dots,E^{(l_{q})}(S)].
\end{equation}
Note that if $\lambda\ne0$ then $\n(S)$ is both the nilradical and the derived algebra of 
$\h(\alpha, \lambda,S)$ (if $\lambda=0$ then $\h(\alpha, \lambda,S)$ is nilpotent and
its derived algebra has codimension 2).
The following result that 
will help us determine the nilpotency degree of $\n(S)$.

\begin{prop}\label{algraduada}
Let  $L=L(1)\oplus L(2)\oplus L(3)\cdots$ be a graded algebra (not necessarily Lie or associative). For each subalgebra
$M$ of $L$ consider the filtration
 $$M=M_1\supseteq M_2\supseteq M_3\supseteq\dots,$$
where
$$
M_i=M\cap (L(i)\oplus L(i+1)\oplus L(i+2)\oplus\cdots),$$
and let $gr(M)$ be the associated graded algebra. Suppose $P$ is a subalgebra of $L$
generated by a set $I$ of homogeneous elements and let $x\mapsto x'$ be a bijection
from $I$ onto a subset $J$ of $L$ such that if $x\in I\cap L(i)$ then
$$
x' \in L_i\text{ and } x'\equiv x\mod L_{i+1}.
$$
Then $P$ is isomorphic to a subalgebra of the graded algebra to gr$(R)$, where $R$ is the subalgebra of $L$
generated by $J$. Moreover, if $L$ is nilpotent then the nilpotency degree of $P$ is less than or equal to
that of $R$. 
\end{prop}

\begin{proof} Let $*$ the underlying operation on $A$. We use Reverse Polish notation and write $ab*$ for the product of $a$ and $b$ in $L$. Consider the  sets $\Lambda(1),\Lambda(2),\dots$ inductively defined as follows: $\Lambda(1)=I\setminus\{0\}$ and if $i>1$ then $\Lambda(i)$ consist of all sequences  $(s,t,*)$, where $s\in\Lambda(j)$, $t\in\Lambda(k)$ and $j+k=i$. We let $\Lambda$ be the
union of all $\Lambda(i)$. Given $s\in\Lambda$, we set $\deg(s)$ to be the sum of the degrees of all elements of $I$ appearing in $s$. For $i\geq 1$, we let $\Omega(i)$ consist of all $s\in\Lambda$ 
having degree $i$. It is clear that every $s\in\Lambda$ 
produces an element of $P$, also denoted by $s$. On the other hand, $s$ gives rise to a sequence $s'$ obtained by replacing every $x\in I$ appearing in $s$ by $x'$,
and such $s'$ produces an element of $R$, also denoted by $s'$.

Now $P\cap L(i)$ (resp.  $R_i/R_{i+1}$) is spanned by all $s$ (resp. $s'+R_{i+1}$)
such that $s\in\Omega(i)$. Moreover, for $s\in \Omega(i)$ we have
$$
s\equiv s'\mod L_{i+1}.
$$
It follows that if $s_1,\dots,s_n\in\Omega(i)$ and $a_1,\dots,a_n\in\mathbb{F}$
then
$$
a_1s_1+\cdots+a_n s_n=0\Leftrightarrow 
   a_1s_n'+\cdots+a_n s_n' \in R_{i+1}.
$$
This shows that the map $\Delta_i:P\cap L(i)\to R_i/R_{i+1}$ given by
\begin{equation}
\label{bocha}
  a_1s_1+\cdots+a_n s_n\mapsto 
   a_1s_1'+\cdots+a_n s_n'+ R_{i+1},\quad  s_i\in\Omega(i), a_i\in\mathbb{F},
\end{equation}
is a well-defined linear monomorphism. Let $\Delta:P\to\text{gr}(R)$ be the
corresponding linear monomorphism and let us write $\times$ for the operation on $\text{gr}(R)$.
We claim that $\Delta(uv*)=\Delta(u)\Delta(v)\times$
for all $u,v\in P$.
We may reduce the verification to the case when $u,v$
are homogeneous and, in fact, to the case when $u\in\Omega(i)$ and $v\in\Omega(j)$.
In this case, $\Delta(uv*)=u'v'\!*\,+\,R_{i+j+1}=\Delta(u)\Delta(v)\times$, as required.

Finally, assume $L$ nilpotent and suppose that there is $s\in\Lambda$ involving $n$ elements of $I$
such that $s\neq 0$.  If $s\in\Omega(i)$ then $\Delta(s)=s'+R_{i+1}\in  R_i/R_{i+1}$
is non-zero, whence $s'\neq 0$. We infer that the nilpotency degree of $R$ 
is greater than or equal
to that of $P$.
\end{proof}

The following proposition shows that $\n(S)$ is independent of $\alpha$ and $\lambda$.
Moreover, it will help us determine the nilpotency of $\n(S)$ from that of $\n(C)$.

\begin{prop}\label{prop.struture}
The Lie algebra $\h(\alpha, \lambda,S)$ is a solvable 
(nilpotent if $\lambda=0$) Lie subalgebra of $\gl(d)$.
Additionally
\begin{enumerate}
\item  $\h(\alpha, \lambda,S)$ is the semidirect product
$\h(\alpha, \lambda,S) = \F D(\alpha, \lambda) \ltimes \n(S)$.
\item  $\n(S)$ is graded by block-degree
and filtered by diagonal-degree.
\item  $\n(C)$ is graded by both block-degree and diagonal-degree. 

\item $\n(C)$  is isomorphic to a subalgebra of the associated graded Lie algebra
$\text{gr}\,(\n(S))$ corresponding to the filtration given by diagonal-degree.

\item The nilpotency degree of $\n(S)$ is at least that of $\n(C)$.
\end{enumerate}
\end{prop}
\begin{proof}
(1): It is not difficult to see that,
for $l\ge 1$,
\begin{equation*}
\big(\text{ad}_{\gl(d)}\,D(\alpha, \lambda)-\lambda  1_{\gl(d)}\big)^l E(S)
=E^{(l)}(S)
\end{equation*}
and thus, the Lie subalgebra of $\h(\alpha, \lambda,S)$ generated by
\[
\{(\text{ad}_{\gl(d)}\, D(\alpha, \lambda))^l E(S):l\ge 0\},
\]
which is invariant under the action of
$\text{ad}\, D(\alpha, \lambda)$,
coincides with $\n(S)$.
Finally, since
$\F D(\alpha, \lambda) \oplus \n(S)$
is a Lie subalgebra of $\h(\alpha, \lambda,S)$
containing $D(\alpha, \lambda)$ and $E(S)$,
it follows that
$\h(\alpha, \lambda,S) = \F D(\alpha, \lambda) \ltimes \n(S)$.

(2) and (3): These are straightforward.

(4)  and (5) These are consequences of Proposition \ref{algraduada} obtained by letting $L$ to be the subalgebra of $\gl(d)$
of all strictly upper triangular matrix; $L(i)=\mathcal{D}_i$ as defined in (\ref{eq.grad1}); $I$ is the set of all 
$E^{(l)}(C)$, $l\geq 0$, and $E^{(l)}(C)'=E^{(l)} (S)$.
\end{proof}

\section{The nilpotency degree of \texorpdfstring{$\n(S)$}{n(S)}}\label{sec.nilp_degree}

If $\g$ is a Lie algebra, let $\{\g^i:i\ge 0\}$
be the lower central series, that is $\g^0=\g$
and $\g^{i+1}=[\g,\g^i]$. 

We will see that, for generic $\vec{d}=(d_1,\dots,d_{k})$ and $S=(S(1),\dots,S(k-1))$, the nilpotency degree of
$\n(S)$ is $k-1$. 
Only very few exceptions occur.
We will see in Theorem \ref{thm.main2} that if $S$ is normalized, then the nilpotency degree of $\n(S)$ is less than $k-1$
if and only if $\vec{d}$  is \emph{odd-symmetric} (as defined below) with $d_1=d_{k}=1$ and $\phi(E(S))=E(S)$
(see \S\ref{subsec.gradings}).

\begin{definition}\label{def.odd-symm}
Given  $\vec{d}=(d_1,\dots,d_{k})$, we say that $\vec d$
is \emph{symmetric} if $d_i=d_{k+1-i}$ for all $i=1,\dots,k$.
We say that $\vec{d}$ is \emph{odd-symmetric} if, in addition,
$k$ is odd and $d_{(k+1)/2}$ is odd.
Also, if $S=(S(1),\dots,S({k-1}))$ is a sequence satisfying \eqref{eq.Condition_S}, we say that  $S$ is
\emph{$\phi$-invariant} if
$E(S)$ is invariant by the automorphism $\phi$. This implies that  $\vec d$
is symmetric. Conversely, if  $\vec d$
is symmetric then the canonical sequence  (see Definition \ref{def.Normalized}) is  $\phi$-invariant.
\end{definition}

\begin{prop}\label{ida}
Let $\vec{d}=(d_1,\dots,d_{k})$ be odd-symmetric,
set $d=|\vec{d}|$,  and let
$S=(S(1),\dots,S({k-1}))$ be a $\phi$-invariant sequence satisfying \eqref{eq.Condition_S}.
Then
\[
A_{i,d+1-i}=0,\qquad i=1,\dots, \tfrac{d-1}2,
\]
for all $A\in\h(\alpha, \lambda,S)$.
In particular, if in addition $d_1=1$ then
$p_{1,k}(\h(\alpha, \lambda,S))=0$.
\end{prop}
\begin{proof}  Since
$\phi(D(0,0))=D(0,0)$ and $\phi(E(S))=E(S)$ it follows from (\ref{eq.generators}) that $\phi$ restricts to the identity
map on $\n(S)$.
Therefore, since
$\vec{d}$ is odd-symmetric (and hence $d$ is odd), the definition of $\phi$ implies that all the entries of $A$ in the antidiagonal
must be zero for all $A\in\n(S)$. Since $d_{(k+1)/2}$ is odd, the antidiagonal entries of
$p_{\tfrac{d+1}2,\tfrac{d+1}2}(D(\alpha,\lambda))$
are zero, with the possible exception of the diagonal entry.
\end{proof}

\subsection{The Lie algebra \texorpdfstring{$\n(S)$}{n(S)} associated} to the canonical sequence \texorpdfstring{$S=C$}{S=C}\label{subsec.canonical}
In this subsection we will consider the case $(\alpha,\lambda,S)=(0,0,C)$.
In order to simplify the notation, let
$\h=\h(0,0,C)$ and $E=E(C)$.

Associated to the Lie algebra $\h$ we define, for $1\le i<j\le k$, the numbers
$$
r_{i,j}= \begin{cases}
0, & \text{if $p_{i,j}(X)=0$ for all $X\in \h$;} \\
 {\min}\{\rk(p_{i,j}(X)):X\in \h, p_{i,j}(X)\neq 0\}, &\text{otherwise.}
 \end{cases}
 $$

\begin{prop}\label{prop.r012}
$r_{i,j}\in \{0,1,2\}$.
\end{prop}

\begin{proof}
It follows from the definition of $E$ that $r_{i,i+1}=1$,
for $1\le i\le k-1$.
For $l\ge 1$, $r_{i,i+l+1}\le 2$ is a consequence of the following two facts.
First, if $X$ is any element of block-degree $l$, then $\rk(p_{i,i+l+1}([E,X]))\le 2$, since all the elements of $p_{i,i+l+1}([E,X])$ are zero, with the possible exception of those in the first column and the last row.

On the other hand, set $j=i+l+1$, we will prove that
if $p_{i,j}([E,X])=0$ for all $X\in\h$,
then $r_{i,j}=0$.
By induction we will show that
\[
p_{i,j}([\ad(D)^{r}E,X])=0,\qquad r\ge0;\; X\in\h.
\]
The case $r=0$ is given. Moreover, given the case $r$,
\begin{align*}
p_{i,j}([\ad(D)^{r+1}E,X])
& =p_{i,j}([D,\ad(D)^{r}E],X]) \\
& =-p_{i,j}([\ad(D)^{r}E,[D,X]]) + p_{i,j}([D,[\ad(D)^{r}E,X]]) \\
& =p_{i,i}(D) p_{i,j}([\ad(D)^{r}E,X])-
p_{i,j}([\ad(D)^{r}E,X])p_{j,j}(D) \\
&=0.
\end{align*}
Since we know from Proposition \ref{prop.struture} that
the elements $\ad(D)^{r}E$, $r\ge0$, generate $\n$,
it follows that $r_{i,j}=0$.
\end{proof}

\begin{prop}\label{prop.F0a}
If $A\in\gl(d)$ has the property
\[
\big(p_{i,j}(A)\big)_{a,b}=
\begin{cases}
1, & \text{if $a,b=a_0,b_0$;}\\
0, & \text{otherwise;}
\end{cases}
\]
then the entries of $p_{i,j}(\ad(D)^k(A))$ are zero
except those contained in the diagonal $b-a=b_0-a_0+k$, in which case:
\[
\big(p_{i,j}(\ad(D)^k(A))\big)_{a_0-i,b_0+k-i}=(-1)^{k-i}\binom{k}{i}.
\]
In particular, $\big(p_{i,j}(A)\big)_{d_i,1}=1$ then
all the entries of $p_{i,j}(\ad(D)^{d_i+d_j-1}(A))$
are zero except
\[
\big(p_{i,j}(\ad(D)^{d_i+d_j-2}(A))\big)_{1,d_j}=(-1)^{d_j-1}\binom{d_i+d_j-2}{d_i-1}.
\]
\end{prop}
\begin{proof}
This is a straightforward computation.
\end{proof}

\begin{prop}\label{prop.F0b}
If there is
$X\in \h$ such that $\big(p_{i,j}(X)\big)_{d_i,1}\ne 0$, then $r_{i,j}=1$.
\end{prop}
\begin{proof}
This is consequence of Proposition \ref{prop.F0a}.
\end{proof}

\begin{prop}\label{prop.dominante}
If $r_{i,j}=1$ then there exists $X\in \h$ such that
\begin{equation}\label{eq.dominante1}
p_{i,j}(X)=
\begin{pmatrix}
0 & \dots & 0 & 1 \\
0 & \dots & 0 & 0 \\
 \vdots & \vdots & \vdots & \vdots \\
 0 & \dots & 0 & 0
\end{pmatrix}
\end{equation}
and if $r_{i,j}=2$ then there exists $X\in \h$ such that
\begin{equation}\label{eq.dominante2}
p_{i,j}(X)=
\begin{pmatrix}
0 & \dots & 1 & * \\
0 & \dots & 0 & 1 \\
 \vdots & \vdots & \vdots & \vdots \\
 0 & \dots & 0 & 0
\end{pmatrix}.
\end{equation}
\end{prop}
\begin{proof}
Let $X\in\h$ be such that $\rk(p_{i,j}(X))=r_{i,j}$,
and let
\begin{align*}
t_0&=\min\{t=b-a:\big(p_{i,j}(X)\big)_{a,b}\ne 0\},\\
T_0& =\{(a,b):b-a=t_0\text{ and } \big(p_{i,j}(X)\big)_{a,b}\ne 0\}.
\end{align*}

If $r_{i,j}=1$ then there is only one pair $(a_0,b_0)\in T_0$.
If $k_0=d_j-1-t_0$ then it follows from
Proposition \ref{prop.F0a} that
$\ad(D)^{k_0}(X)$ is, up to a non-zero scalar, as stated.

If $r_{i,j}=2$ then there are at most two possible pairs $(a,b)\in T_0$.
It follows from Proposition \ref{prop.F0a} that,
if $k_0=d_j-2-t_0$, then
the only possible non-zero entries of
$p_{i,j}(\ad(D)^{k_0}(X))$  are
\begin{align*}
\big(p_{i,j}(X)\big)_{1,d_j-1}\quad  & \big(p_{i,j}(X)\big)_{1,d_j} \\
& \big(p_{i,j}(X)\big)_{2,d_j}.
\end{align*}
Moreover, the pair
\begin{equation}\label{eq.pair}
\Big(\big(p_{i,j}(X)\big)_{1,d_j-1},\; \big(p_{i,j}(X)\big)_{2,d_j}\Big)
\end{equation}
is a linear combination of two pairs of two consecutive
binomial numbers $\binom{k_0}{l}$, $0\le l\le k$, that is
\[
\Big(
 \big(p_{i,j}(X)\big)_{1,d_j-1},\;
\big(p_{i,j}(X)\big)_{2,d_j}
\Big)
=
x_1
\Big(
\binom{k_0}{l_1} ,\;
\binom{k_0}{l_1+1}
\Big)
+
x_2
\Big(
\binom{k_0}{l_2} ,\;
\binom{k_0}{l_2+1}
\Big)
\]
with $0\le l_1\ne l_2\le k_0$ for some $(x_1,x_2)\ne(0,0)$.
Since
$
\Big(
\binom{k_0}{l_1} ,\;
\binom{k_0}{l_1+1}
\Big)
$ and $
\Big(
\binom{k_0}{l_2} ,\;
\binom{k_0}{l_2+1}
\Big)
$  are linearly independent,
it follows that the pair \eqref{eq.pair} is non-zero.
Finally, we conclude that
$\ad(D)^{k_0}(X)$ is, up to a non-zero scalar, as stated because otherwise
we would have $r_{i,j}=1$.
\end{proof}

\begin{prop}\label{prop.F1} If there exists $X\in\h$ such that either
\[
p_{i,j}(X)=0\quad\text{ and }\quad p_{i+1,j+1}(X)\neq0
\]
or
\[
p_{i,j}(X)\ne0\quad\text{ and }\quad p_{i+1,j+1}(X)=0
\]
then $r_{i,j+1}=1$. Moreover, any of the following:
\begin{enumerate}
    \item[(a1)]  $r_{i,j}=0$ and $r_{i+1,j+1}\neq0$,
    \item[(a2)]  $r_{i,j}\ne 0$ and $r_{i+1,j+1}=0$,
    \item[(b)]  $r_{i,j}=r_{i+1,j+1}=1$ and $d_{i}\ne d_{j+1}$,
    \item[(c1)]  $r_{i,j}=1$, $r_{i+1,j+1}=2$ and $d_{i}+1\ne d_{j+1}$,
    \item[(c2)]  $r_{i,j}=2$, $r_{i+1,j+1}=1$ and $d_{i}\ne d_{j+1}+1$.
\end{enumerate}
implies the existence of such an $X$ and thus $r_{i,j+1}=1$.
\end{prop}
\begin{proof}
Suppose first there exists $X\in\h$ such that $p_{i,j}(X)\neq 0$ and $p_{i+1,j+1}(X)=0$. Repeatedly bracketing $X$ with $D$
and applying \cite[Lemma 6.1]{CGS1}, we may assume without loss of generality that the first row of $p_{i,j}(X)$ is not zero. It is then clear that
${\rm rk}(p_{i,j+1}([E,X]))=1$. The case  $p_{i,j}(X)=0$ and $p_{i+1,j+1}(X)\neq 0$ is analogous.

Now we prove the particular statements.
By symmetry, it is enough to prove (a1), (b) and (c1).

Proof of (a1): it is immediate that (a1) implies the existence of
$X\in\h$ such that $p_{i,j}(X)=0$  and $p_{i+1,j+1}(X)\neq0$.

Proof of (b): let $X\in \h\cap\mathcal{D}_{t_X}$ be homogeneous
such that all the entries of $p_{i,j}(X)$ are zero except that
$\big(p_{i,j}(X)\big)_{1,d_j}=1$, as granted by  Proposition \ref{prop.dominante}.
This implies that $t_X=d_{j-1}+\dots+d_i+(d_j-1)$.

Similarly, let $Y\in \h\cap\mathcal{D}_{t_Y}$ be homogeneous
such that all the entries of  $p_{i+1,j+1}(Y)$ are zero except that
$\big(p_{i+1,j+1}(Y)\big)_{1,d_{j+1}}=1$.
Now $t_Y=d_{j}+\dots+d_{i+1}+(d_{j+1}-1)$.

It follows from the hypothesis that
\[
t_Y-t_X=d_{j+1}-d_i\ne0.
\]
Therefore, either $t_Y>t_X$,
in which case $p_{i,j}(Y)=0$
or
$t_X>t_Y$,
in which case $p_{i+1,j+1}(X)=0$, and we are done.

Proof of (c1): This is analogous to the proof of (b).
\end{proof}

\begin{prop}\label{useful}
 If $r_{i,j}=1$ and one of the following hold:
\begin{enumerate}
    \item[(a)]  $d_i,d_j>1$,
    \item[(b1)]  $d_j>1$ and $r_{i+1,j+1}=0$,
    \item[(b2)]  $d_i>1$ and $r_{i-1,j-1}=0$,
    \item[(c)]  $r_{i+1,j+1}=r_{i-1,j-1}=0$.
\end{enumerate}
then $r_{i-1,j+1}=1$.
\end{prop}
\begin{proof}
Any of these conditions implies that, for any $X\in \h$,
\[
\big(p_{i-1,j+1}([[X,E],E]]\big)_{d_{i-1},1}=-2\big(p_{i,j}(X)\big)_{1,d_j}.
\]
Since $r_{i,j}=1$, it follows from Proposition \ref{prop.dominante} that there exists
 $X\in \h$ such that $\big(p_{i,j}(X)\big)_{1,d_j}\neq 0$, and thus $\big(p_{i-1,j+1}([[X,E],E]]\big)_{d_{i-1},1}\neq 0$.
 Now Proposition \ref{prop.F0b} implies $r_{i-1,j+1}=1$.
\end{proof}
\begin{prop}\label{adj}
 If $r_{i,j}=2$, then $r_{i+1,j}\neq 2$ and $r_{i,j-1}\neq 2$.
\end{prop}
\begin{proof}
 We suppose the result is false and select $(i,j)$ such that $r_{i,j}=2$ as well as $r_{i+1,j}=2$ or  $r_{i,j-1}=2$,
with $j-i$ as small as possible. Since  $r_{s,s+1}=1$ for all $s$, we have $j-i>2$.
We may assume without loss of generality that $i=1,j=k$ and so $r_{1,k}=2$, with $r_{1,k-1}=2$ or $r_{2,k}=2$.

Suppose first $r_{1,k-1}=2$. By the minimality of $k-1$, we have
$r_{1, k-2},r_{2, k-1}\neq 2$ and, since $r_{1,k-1}=2$,
Proposition \ref{prop.F1} (a1), (a2), implies that
\[
r_{1, k-2},r_{2, k-1}=1.
\]
 Since $r_{1,k-1}=2$ then $d_{k-1}>1$ and hence, since $r_{2,k-1}=1$ and $r_{1,k}=2$,
 Proposition \ref{useful} (a) implies $d_2=1$ and thus $r_{2,k}\ne 2$.
 Since $r_{1,k}=2$, Proposition \ref{prop.F1} (a2) implies  $r_{2,k}=1$.

 Now we have  $r_{3,k}\neq 0$ since, otherwise,
 Proposition \ref{useful} (b1), applied to $(i,j)=(2,k-1)$ would imply that $r_{1,k}=1$.
 Moreover, we claim $r_{3,k}=2$.

 If $r_{3,k}=1$ we can find a homogeneous $X\in \h\cap\mathcal{D}_t$
 such that $p_{3,k}(X)$ is as stated in Proposition \ref{prop.dominante}, that is
 $\big(p_{3,k}(X)\big)_{1,d_k}=1$.
 Since $1=d_2<2\le d_k$, Proposition \ref{prop.degrees} implies $p_{2,k-1}(X)=0$.
 Since $r_{1,k-1}=2$, Proposition \ref{prop.F1} implies $p_{1,k-2}(X)=0$.
 Therefore
 \[
 p_{1,k-1}([X,E])=0\quad\text{ and }\quad p_{2,k}([X,E])\ne 0
 \]
 and, once again,  Proposition \ref{prop.F1} implies $r_{1,k}=1$, a contradiction.
 We have proved that $r_{3,k}=2$ and hence $d_3\ge 2$,
 $r_{3,k-1}\neq 2$ by the minimality of $k-1$,
 and $r_{3,k-1}\neq 0$ by Proposition \ref{prop.F1}.
 Therefore $r_{3,k-1}=1$ and, it follows from Proposition \ref{prop.dominante}
 that there is a homogeneous  $Y\in \h$ as in \eqref{eq.dominante1},
 that is with  $\big(p_{3,k-1}(Y)\big)_{1,d_{k-1}}=1$.
 Taking into account that $d_3,d_{k-1}\ge 2$ it is not difficult to see that
\[
\big(p_{1,k}([[[Y,E],E],E]))\big)_{d_1,1}=3
\]
which by Proposition \ref{prop.F0b} implies that $r_{1,k}=1$, a contradiction.

Suppose next $r_{2,k}=2$. Let $\h'=\phi(\h)$, the Lie algebra associated to $\phi(\vec d)=(d_k,\dots,d_1)$, $\phi(C)$ and $\phi(D)$,
and denote by $r'_{i,j}$ the corresponding ranks. Then $r_{1,k}=r_{2,k}=2$ imply $r'_{1,k}=r'_{1,k-1}=2$, which was shown above to be impossible.
\end{proof}

\begin{prop}\label{sym1}
$r_{i,j}=2$ implies $d_i=d_j>1$.
\end{prop}
\begin{proof}
Since $r_{i,j}=2$, we have $d_i,d_j>1$. We must show that $d_i=d_j$.
 We know by Proposition \ref{adj} that $r_{i,j-1}, r_{i+1,j}\neq 2$.
 Also, by Proposition \ref{prop.F1}, it follows that $r_{i,j-1}, r_{i+1,j}\neq 0$.
 Therefore $r_{i,j-1}= r_{i+1,j}=1$.
 Now Proposition \ref{prop.F1} (b) implies $d_i=d_j$.
\end{proof}
\begin{prop}\label{sym2}
If $r_{u,v}=0$ then $d_u=1$ or $d_v=1$.
\end{prop}
\begin{proof}
We suppose the result is false and select $(u,v)$ such that $r_{u,v}=0$, $d_u>1$ and $d_v>1$,
with $v-u$ as small as possible.
We may assume without loss of generality that $u=1,v=k$ and so $d_1,d_k>1$. We claim that $r_{1,k-1}=r_{2,k}=0$.

\smallskip
\noindent
Case  $r_{1,k-1}=0$, $r_{2,k}\ne 0$; or  $r_{1,k-1}\ne 0$, $r_{2,k}=0$:
Impossible by Proposition \ref{prop.F1}.

\smallskip
\noindent
Case  $r_{1,k-1}=r_{2,k}=1$: It follows from Proposition \ref{prop.F1} (b) that
$d_1=d_k$ and it is clear that if $d_1\neq 1$ then $r_{1,k}\neq 0$, thus $d_1=1$.

\smallskip
\noindent
Case $r_{1,k-1}=r_{2,k}=2$: This implies that $d_1,d_2,d_{k-1},d_k\ge 2$.
Consider $r_{2,k-1}$. It is not 0 by Proposition \ref{prop.F1} and it
 cannot be 2 by Proposition \ref{adj}.
Hence $r_{2,k-1}=1$ and now Proposition \ref{useful} implies $r_{1,k}=1$
contradicting our hypothesis.

\smallskip
\noindent
Case $r_{1,k-1}=2, r_{2,k}=1$:
It follows from Propositions \ref{adj} and \ref{prop.F1} (a1), (a2) that $r_{2,k-1}=1$.
Since $r_{1,k-1}=2$, we have $d_1=d_{k-1}>1$ by Proposition \ref{sym1}. We cannot have $d_2>1$, for otherwise
$r_{1,k}=1$ by Proposition \ref{useful}. Thus $d_2=1$. Since $r_{2,k}=1$, Proposition \ref{prop.dominante}
ensures the existence of $X\in\h$ such that $p_{2,k}(X)=(0,\dots,0,1)$. If $d_k>1$ then $p_{1,k}([E,X])\neq 0$.
But $r_{1,k}=0$, so $d_k=1$.

\smallskip
\noindent
 Case $r_{1,k-1}=1, r_{2,k}=2$: This case is symmetric to the one above.

\smallskip

This proves the claim that $r_{1,k-1}=r_{2,k}=0$. By the minimality of $k-1$ and the fact that $d_1,d_k>1$, we infer
$d_2=d_{k-1}=1$.

Let $j_0$ be the largest $j$ such that
\[
r_{i,k-j+i}=0\text{ for all $i=1,\dots,j$}.
\]
Clearly $2\le j_0\le k-2$ and, again,
the minimality of $k-1$ implies
\begin{equation}\label{eq.d=1}
d_j=d_{k+1-j}=1\text{ for all $2\le j\le j_0$}.
\end{equation}
Since, by definition of $j_0$, we have
$r_{i,k-(j_0+1)+i}\ne0$ for some $i$,
it follows from Proposition \ref{prop.F1}
(a1) or (a2) that in fact
$r_{i,k-(j_0+1)+i}\ne 0$
for all $i=1,\dots,j_0+1$.
Moreover, \eqref{eq.d=1} implies that
\[
r_{i,k-(j_0+1)+i}=1\text{ for all $i=2,\dots,j_0$}.
\]
 Let
$X\in \h\cap\bar{\mathcal{D}}_{k-(j_0+1)}$, $X\ne0$
such that $[D,X]=0$. Since $r_{i,k-j_0+i}=0$ for all $i=1,\dots,j_0$, we must have
$[E,X]=0$ and thus $[E+D,X]=0$. But $D+E$ is a Jordan block, so $X$ is a polynomial in $E+D$, that is,
$X=X_1+\cdots+X_{d-1}$, where $X_i\in\mathcal{D}_i$ and all non-zero entries of $X_i$, if any, are identical.
But block $(2,k+1-j_0)$ has size $1\times 1$ and $X\in\bar{\mathcal{D}}_{k-(j_0+1)}$, so all entries of $X$ in row 2
outside of column $k+1-j_0$ are equal to 0. Thus $X=X_{k-(j_0+1)}=a(E+D)^{k-(j_0+1)}$ for some $a\neq 0$, and therefore $(E+D)^{k-(j_0+1)}\in\bar{\mathcal{D}}_{k-(j_0+1)}$.
But $E\in\bar{\mathcal{D}}_{1},D\in\bar{\mathcal{D}}_{0}$, so $(E+D)^{k-(j_0+1)}=E^{k-(j_0+1)}$. If $e_1,\dots,e_d$ is the canonical
basis of $\C^d$, then $(E+D)^{k-(j_0+1)}e_d=e_{d-(k-(j_0+1))}\neq 0$. But $d_k>1$ implies $E^{k-(j_0+1)}e_d=0$.
\end{proof}

Now we can prove the crucial step.
\begin{prop}\label{crucial} Let $k\ge 2$ and $\vec{d}=(d_1,\dots,d_{k})$.
Then
 \begin{enumerate}
\item \label{sym3}  If $r_{2,k-1}=2$ and $d_1=d_k=1$ then $r_{1,k}=0$.
\item \label{sym4}  If $r_{1,k}=0$ then $d_1=1$ and $d_k=1$.
\item \label{cru} If $r_{2,k-1}=1$ then $r_{1,k}=1$, unless $k=4$ and $\vec{d}=(1,1,1,1)$.
\item \label{vuelta2} If $r_{1,k}=2$ then $k$ is odd
and $\vec{d}$ is odd-symmetric  with $d_1=d_{k}>1$.
\item \label{vuelta0} If $r_{1,k}=0$ then either $k$ is even and $\vec{d}=(1,\dots,1)$, or  $k$ is odd
and $\vec{d}$ is odd-symmetric with $d_1=d_{k}=1$.
\end{enumerate}
 \end{prop}
\begin{proof}
 We use induction on $k$.  For $k= 2$ there is nothing to prove.
 We assume $k\ge3$ and that the whole proposition is true for lower values of $k$.

\medskip
 \noindent
 Proof of part \eqref{sym3}, we have $r_{2,k-1}=2$ and $d_1=d_k=1$:
 By induction hypothesis on part \eqref{vuelta2},
 $r_{2,k-1}=2$ implies that $k-2$ is odd
 and $\vec{d}$ is odd-symmetric.
Proposition \ref{ida} and $d_1=d_k=1$ imply  $r_{1,k}=0$.

\medskip
 \noindent
 Proof of part \eqref{sym4}, we have
 $r_{1,k}=0$: As in   the proof of Proposition \ref{sym2},
 we will consider all possible values for $r_{1,k-1},r_{2,k}$.

The cases  $r_{1,k-1}=0$,  $r_{2,k}\ne 0$ and
$r_{1,k-1}\ne 0$, $r_{2,k}=0$ are impossible by  Proposition \ref{prop.F1}.

The case $r_{1,k-1}=r_{2,k}=0$ follows by induction hypothesis
on part \eqref{sym4}.

The case  $r_{1,k-1}=r_{2,k}=1$ is as in the proof of
Proposition \ref{sym2} in which we obtained $d_1=d_k=1$.

The case $r_{1,k-1}=r_{2,k}=2$ implies $d_1, d_k> 1$, which contradicts
Proposition \ref{sym2}.

Finally, let us prove that the case $r_{1,k-1}=2$, $r_{2,k}=1$ is impossible
(the symmetric case $r_{1,k-1}=1$, $r_{2,k}=2$ will then also be impossible).

This case
implies that $d_{k-1},d_1\ge 2$ and thus, by Proposition \ref{sym2}, $d_k=1$.
 Proposition \ref{prop.F1} (c2) implies that $d_1= 2$.
 Proposition \ref{adj} implies $r_{2,k-1}\ne 2$,
 Proposition \ref{prop.F1} implies $r_{2,k-1}\ne 0$, and thus  $r_{2,k-1}=1$.
 Since  $d_{k-1}\ge 2$, if  $d_2>1$, Proposition \ref{useful} (a) would imply
 that $r_{1,k}=1$; hence $d_2=1$.
 Let $ l\ge3$ be the first index such that $d_{l-1}=1$ but $d_{l}>1$.
 Thus we have
 \[
 2=d_1,\; 1=d_2=\dots =d_{l-1},\; 2\le d_l,\;\dots,\; 2\le d_{k-1},\;1=d_k.
 \]
  Since $d_{k-1}\geq 2$, the $k$th block column of $E^2$ is zero. It follows that the $k$th block columns of $E^2,E^3,\dots$ are zero.
	Likewise, since $d_1=2$, $d_2=\dots=d_{l-1}=1$ and $d_l\geq 2$, the first block columns of $E^l,E^{l+1},\dots$ are zero. We will use these
comments to compute powers of $\ad(E)=L_E-R_E$, where $L_E(Y)=EY$ and $R_E(Y)=YE$ for all $Y\in M_d$, by means of the binomial expansion.

Now we will show that $r_{l,k-1}\ne 0,1,2$, which is a contradiction.

  Since $d_l,d_{k-1}\ge 2$,  Proposition \ref{sym2} implies $r_{l,k-1}\ne 0$.
    We next show that $r_{l,k-1}\ne 1$. If not, there is
  $X\in\h$ such that $\rk(p_{l,k-1}(X))=1$ and by Proposition \ref{prop.dominante}
  we may assume as in \eqref{eq.dominante1}. The above comments on powers of $E$ give
	$$
	p_{1,k}(\ad(E)^{l}(X))=(-l)p_{1,l}(E^{l-1})p_{l,k-1}(X)p_{k-1,k}(E)=\begin{pmatrix} 0\\ -l\end{pmatrix},
	$$
which contradicts $r_{1,k}=0$.

Finally, let us show that $r_{l,k-1}\ne 2$. Suppose, if possible, that $r_{l,k-1}=2$. Since $l\geq 3$, the induction hypothesis  on part \eqref{sym3}
applies to give $r_{l-1,k}=0$.

Since $d_k = 1$ we have $r_{l,k}\in\{0,1\}$.
Similarly, $d_{l-1} = 1$ implies $r_{l-1,k-1}\in\{0,1\}$.
By the induction hypothesis on part \eqref{sym4},
 $d_{k-1} > 1$ implies  $r_{l-1,k-1}\ne 0$ and
 $d_{l} > 1$ implies  $r_{l,k}\ne 0$.
  Therefore $r_{l,k}=r_{l-1,k-1}=1$ and thus
 there is $X\in\h$ such that such that $\rk(p_{l,k}(X))=1$, but since
  $r_{l-1,k}=0$, Proposition \ref{prop.F1} implies that
  $\rk(p_{l-1,k-1}(X))=1$.
  By Proposition \ref{prop.dominante} we may assume that
  $p_{l,k}(X)$ is an in \eqref{eq.dominante1} and that, up to a scalar multiple,
   $p_{l-1,k-1}(X)$ is also
  as in \eqref{eq.dominante1}. But given that $r_{l-1,k}=0$, we must have $p_{l-1,k}([E,X])=0$, which implies that
  $p_{l-1,k-1}(X)$ is exactly as in \eqref{eq.dominante1}. On the other hand, the above comments on powers of $E$ give
  	$$
	p_{1,k}(\ad(E)^{l-1}(X))=p_{1,l}(E^{l-1})p_{l,k}(X)-(l-1)p_{1,l-1}(E^{l-2})p_{l-1,k-1}(X)p_{k-1,k}(E).
	$$
  Therefore, our descriptions of $p_{l,k}(X)$ and $p_{l-1,k-1}(X)$ yield
   $$
	p_{1,k}(\ad(E)^{l-1}(X))= \begin{pmatrix} 0\\ 1\end{pmatrix}-(l-1)\begin{pmatrix} 0\\ 1\end{pmatrix}=\begin{pmatrix} 0\\ 2-l\end{pmatrix}\neq \begin{pmatrix} 0\\ 0\end{pmatrix},
$$
which contradicts $r_{1,k}=0$.

\medskip
 \noindent
 Proof of part \eqref{cru}, we have  $r_{2,k-1}=1$:
 If $d_1\ne d_k$,
 it follows from Proposition \ref{sym1}
 and part \eqref{sym4}
 that $r_{1,k}=1$.
 Therefore, we assume from now on $d_1=d_k$.
Let us consider now $r_{1,k-1}$ and $r_{2,k}$.

Suppose first $r_{1,k-1}=r_{2,k}=0$. If $k$ is odd the
inductive hypothesis
on part (5) implies $(d_1,...,d_{k-1})=(1,...,1)$ and
$(d_2,...,d_k)=(1,...,1)$.
If $k$ is even the inductive hypothesis on part (5) implies that
$(d_1,...,d_{k-1})$ and $(d_2,...,d_k)$ are odd-symmetric,
$d_1=d_{k-1}=1$ and
$d_2=d_k=1$.
Then, repeatedly using both odd-symmetries, we obtain
$d_{k-2}=d_{k-3}=\dots=d_4=d_3=1$.
Thus $\vec{d}=(1,...,1)$ regardless of
the parity of $k$, whence $r_{i,j}=0$
whenever $j-i>1$.
Since $r_{2,k-1}=1$, we infer $k=4$.

  The cases  $r_{1,k-1}=0$,  $r_{2,k}\ne 0$ and
$r_{1,k-1}\ne 0$, $r_{2,k}=0$ imply that $r_{1,k}=1$ by  Proposition \ref{prop.F1}.

  The case  $r_{1,k-1}=2$ implies that $r_{1,k}$ cannot be 2 by
  Proposition \ref{adj} and that  $r_{1,k}$ cannot be 0
  by part  \eqref{sym4}, and thus $r_{1,k}=1$.
  Similarly, if $r_{2,k}=2$ then $r_{1,k}=1$.

 Therefore, we can assume $r_{1,k-1}=r_{2,k}=1$ and thus $(d_1,\dots,d_k)\ne(1,\dots,1)$.
  There are 3 cases to consider:

{\sc Case 1.} There are $1<i<j<k$ such that $d_i,d_j>1$ and
  \begin{equation}\label{eq.d_equal_1}
 \text{$d_l=1$ for $j<l\le k$ and $1\le l<i$}.
  \end{equation}
  We have $r_{i,j}\neq0$ by Proposition \ref{sym2}.

  Assume first $r_{i,j}=1$ Then there is $X\in\h$
  such that $\rk(p_{i,j}(X))=1$ and, by
  Proposition \ref{prop.dominante}, we may assume $p_{i,j}(X)$
  as in \eqref{eq.dominante1}.
    Since  $d_i,d_j>1$, for 
all $q,\beta$ satisfying $1\leq q<\beta\leq k+i-j-1$ and $q<i$ we have
    \[
    p_{i-q,j+\beta-q}((\text{ad}\, E)^{\beta} X)
       =(-1)^{\beta-q}\binom{\beta}{q}.
    \]
Let $q=i-1$ and $\beta=k+i-j-1$. Then
 $p_{1,k}(\text{ad}(E)^\beta (X))\neq 0$,
 whence $r_{1,k}=1$.

  Now assume $r_{i,j}=2$.
  By the induction hypothesis on  part \eqref{vuelta2} we have
  $j+1-i$ odd and 
  \begin{equation}\label{eq.odd_sym}
  (d_{i},d_{i+1},\dots,d_{j})\quad\text{is odd-symmetric}.
  \end{equation}
The induction hypothesis on part \eqref{sym3} implies
$r_{i-1,j+1}=0$.
We claim that $r_{i-1,j}=r_{i,j+1}=1$.
Indeed, the induction hypothesis on part \eqref{sym4} implies that neither
$r_{i-1,j}$ nor $r_{i,j+1}$ is equal to 0,
and, since $d_{i-1}=d_{j+1}=1$, we have that neither
$r_{i-1,j}$ nor $r_{i,j+1}$ is equal to 2.
Summarizing, we have
\[
 r_{i,j}=2,\qquad r_{i-1,j}=r_{i,j+1}=1,\qquad r_{i-1,j+1}=0.
\]

 Thus there is
  $X\in \h$
  such that $\rk(p_{i-1,j}(X))=1$, but since
  $r_{i-1,j+1}=0$, Proposition \ref{prop.F1} implies that
  $\rk(p_{i,j+1}(X))=1$ and
  $p_{i-1,j+1}([X,E]))=0$.
    By Proposition \ref{prop.dominante} we may assume that
    \[
    p_{i-1,j}(X)=\begin{pmatrix}
     0& \cdots &0&1
    \end{pmatrix}\qquad
    \text{ and }\qquad
    p_{i,j+1}(X)=\begin{pmatrix}
     x\\ 0\\ \vdots \\ 0
    \end{pmatrix}.
    \]
 But since  $p_{i-1,j+1}([X,E]))=0$, we must have $x=1$.
    This implies that for all $q,\beta$ satisfying $1\leq q\leq \beta\leq k+i-j-2$ and $q<i$ we have
    \[
    p_{i-q,j+1+\beta-q}((\text{ad}\, E)^{\beta} X)
    =\pm\left(\binom{\beta}{q}-\binom{\beta}{q-1}\right).
    \]
     Let $q=i-1$ and $\beta=k+i-j-2$. Then
  \[
    p_{1,k}((\text{ad}\, E)^{k-j+i-2} X)
    =\pm\left(
\binom{k-j+i-2}{i-1}-\binom{k-j+i-2}{i-2}
\right).
    \]
If this number is not zero,
 then $\rk\big(p_{1,k}((\text{ad}\, E)^{k-j+i-2} X)\big)=1$
 and thus $r_{1,k}=1$.
 Otherwise
    \[
\binom{k-j+i-2}{i-1}=\binom{k-j+i-2}{i-2}
    \]
and hence $j=k+1-i$.
This, together with \eqref{eq.d_equal_1}
and \eqref{eq.odd_sym}, imply $k$ odd and
 \[
  (d_{2},d_{3},\dots,d_{k-1})\quad\text{is odd-symmetric}.
  \]
Now, Proposition \ref{ida} implies $r_{2,k-1}=0$, a contradiction.

 {\sc Case 2.} There is $1<i<k$ such that $d_i>1$ and $d_l=1$ for all $l\neq i$.

Set
$X=(\text{ad}\, D)^{d_i-1} E\in \h$, so that
    \[ p_{i-1,i}(X)=(0,\dots,0,(-1)^{d_i-1}),\;
    p_{i,i+1}(X)=\begin{pmatrix}
     1\\ 0\\ \vdots \\ 0
    \end{pmatrix},
    \]
and $p_{j,j+1}(X)=0$ in all other cases. It follows that
$$
p_{1,k}((\text{ad}\, E)^{k-2} X)=\pm\left(
\binom{k-2}{i-1}-(-1)^{d_i-1}\binom{k-2}{i-2}
\right).
$$
If $d_i$ is even this number is not zero, so $r_{1,k}=1$. Suppose $d_i$ is odd.
Then the above number is zero if and only if $k$ is odd and $i=(k+1)/2$, in which case
$\vec d$ is odd-symmetric. But then Propositions \ref{ida} and \ref{prop.dominante} imply that
$r_{2,k-1}=0$, a contradiction.

\medskip

  {\sc Case 3.} 
$d_1=d_k>1$ and  $d_j>1$ for some $j=2,\dots, k-1$. 
Let us consider the 
Lie algebra 
$\h'=\h(0,0,C)$ associated to   
\begin{equation}
  \vec d{\,}'= (1,d_2,\dots,d_{k-1},1)
  \end{equation}
  and let $\n(C)'$ be the usual 1-dimensional factor of $\h'$.
  
  Let $f:M_{|\vec d{\,}'|\times |\vec d{\,}'|}\to M_{|\vec d|\times |\vec d|}$
  be the symmetric inclusion (the natural inclusion with respect to the block decomposition). 
  Let us call $D'$ and $E'$ the generators of $\h'$ let 
  $\tilde D=f(D')$ and $\tilde E=f(E)$.
  It is clear that 
  \begin{equation}\label{eq.E'D'}
   E=\tilde E\qquad\text{ and }\qquad  D=\tilde D+D_0
\end{equation}
where $D_0\in \h$ consists of two Jordan blocks of size $d_1=d_k$ located
respectively in the upper-left and lower-right corners.

Recall that, by Proposition \ref{prop.struture}, we know that the matrices 
\[
E^{(l)}=\ad(D)^{l}(E)\qquad\text{ and }\qquad {E'}^{(l)}=\ad(D')^{l}(E')
\]
(with $l\ge 0$) generate  
 $\n(C)$ and $\n(C)'$ respectively.

It follows from \eqref{eq.E'D'} that, for each $l\ge 0$,  
\[
 {E}^{(l)}\equiv f({E'}^{(l)}) \mod \mathfrak{q}
\]
where $\mathfrak{q}$ is the ideal of $\h$ formed by the 
upper-right hook of length $|\vec d|-(d_1-1)$ and width $d_1-1$;
that is,
if $Y\in\mathfrak{q}$ and $Y_{ij}\ne0$ then 
\[
i< d_1\text{ and }j\ge d_1, \quad\text{ or }\quad
i\le  |\vec d|-d_1+1\text{ and }j> |\vec d|-d_1+1. 
\]
By hypothesis, $r_{2,k-1}=1$. This implies that  $r'_{2,k-1}=1$
( $r'_{i,j}$ is the $r_{i,j}$ corresponding to $\h'$).
Indeed
 $r_{2,k-1}=1$ implies that there is a matrix $X\in\h$,
 that is a linear combination of brackets of the $E^{(l)}$'s 
 with $\text{rk}\big(p_{2,k-1}(X)\big)=1$.
 The same linear combination of brackets of the ${E'}^{(l)}$'s 
yields a matrix $X'\in\h'$ such that 
\[
 f(X')\equiv X \mod \mathfrak{q}
\]
and hence $\text{rk}\big(p_{2,k-1}(X')\big)=1$.

 Except in the case $k=4$ and $\vec d{}'=(1,1,1,1)$,
we are now in a position to apply Case 1 or Case 2 to $\h'$
to conclude that $r'_{1,k}=1$. 
We now repeat the argument above in the reverse order to conclude that 
 there is a matrix $X\in\h$ 
 with 
 \[
  \big(p_{1,k}(X)\big)_{d_1,1}=1
 \]
This shows that  $r_{1,k}=1$.

In the case $\vec d{}'=(1,1,1,1)$, that is 
$\vec d=(d_1,1,1,d_4)$, it is straightforward to see that
\begin{equation}
 \text{rk}\big( (\ad\,D)^{2d_1-3} ([[[D,E],E],E]\big)=1.
  \end{equation}
We point out that the hypothesis $r_{2,k-1}=1$ does not hold
when $\vec d=(d_1,1,\dots,1,d_k)$ ($d_1=d_k>1$) and $k\ne 4$.

\medskip
 \noindent
 Proof of parts \eqref{vuelta2} and \eqref{vuelta0},
 we have $r_{1,k}\ne1$:
 It follows from part \eqref{cru}
that $r_{2,k-1}\ne1$ (unless $k=4$ and $\vec d=(1,1,1,1)$, in which case (4) and (5) are automatically true).
We now apply the induction hypothesis on parts \eqref{vuelta2} and \eqref{vuelta0}
and we consider the cases $k$ even and $k$ odd.

If $k$ is even, then  $(d_2,\dots,d_{k-1})=(1,\dots,1)$
(in particular $r_{2,k-1}=0$).
We may assume $d_1\ge d_k$, we will show that $d_1= d_k=1$.

 Let
$X=(\text{ad}\, D)^{d_1-1} E\in \h$, so that
    \[
    p_{1,2}(X)=\begin{pmatrix}
     1\\ 0\\ \vdots \\ 0
    \end{pmatrix}.
    \]
 
   If $d_1> d_k$ then
   $p_{\alpha,1+\alpha}(X)=0$ for all $2\le\alpha\le k-1$
   and it is clear that  $\rk(\text{(ad}\, E)^{k-2} X)=1$ and hence $r_{1,k}=1$,
   a contradiction. Therefore $d_1= d_k$.

If $d_1= d_k>1$ then
$p_{\alpha,\alpha+1}(X)=0$ for all $2\le\alpha\le k-2$ and
 \[
    p_{k-1,k}(X)=\begin{pmatrix}
     0& \cdots &0&(-1)^{d_k-1}
    \end{pmatrix}.
\]
Set $Y=(\text{ad}\, D)^{d_k-2} (\text{ad}\, E)^{k-2} X$. Since $k$ is even, we have
\[
p_{1,k}(Y)=
\begin{pmatrix}
0 & \dots & (-1)^{d_k-2} & 0 \\
0 & \dots & 0 & (-1)^{d_k-1} \\
 \vdots & \vdots & \vdots & \vdots \\
 0 & \dots & 0 & 0
\end{pmatrix}
\]
and hence
$\rk(p_{1,k}([D,Y]))=1$, a contradiction. Therefore $d_1= d_k=1$.

If $k$ is odd, then the  induction hypothesis on
parts \eqref{vuelta2} and \eqref{vuelta0}
implies that $(d_2,\dots,d_{k-1})$ is odd-symmetric.
If $r_{1,k-1}\ne 1$, the
induction hypothesis on parts \eqref{vuelta2} and \eqref{vuelta0}
implies $(d_1,\dots,d_{k-1})=(1,\dots,1)$ and thus $d_k=1$
(otherwise we would obtain $r_{1,k}=1$).
Hence $r_{1,k-1}=1$ and similarly $r_{2,k}= 1$.
Now $r_{1,k}\ne 1$, part \eqref{sym4} and
Proposition \ref{sym1} imply
$d_1=d_k$ and thus $(d_1,\dots,d_{k})$ is odd-symmetric.
 Finally, if $r_{1,k}=2$ then Proposition \ref{sym1} implies $d_1=d_k>1$, while if $r_{1,k}=0$ then $d_1=d_k=1$ by part (2).
\end{proof}

Summarizing, we have proved the following theorem.

\begin{theorem}\label{thm.main} Let $k\ge 2$ and $\vec{d}=(d_1,\dots,d_{k})$.
Then the nilpotency degree of $\n(C)$ is $k-1$ except when $r_{1,k}=0$.
This occurs if and only if
 \begin{enumerate}
\item $\vec{d}=(1,\dots,1)$, in which case $\n(C)$ is 1-dimensional abelian.
\item $k$ is odd, $\vec{d}$ is odd-symmetric  with $d_1=d_{k}=1$,
in which case, if $\vec{d}\ne (1,\dots,1)$, the nilpotency degree is $k-2$.
\end{enumerate}
In addition, $r_{1,k}=2$ if and only if $k$ is odd, $\vec{d}$ is odd-symmetric  with $d_1=d_{k}>1$.
 \end{theorem}
 \begin{proof}
  If $\vec{d}=(1,\dots,1)$, then it is clear that $\n(C)$ is 1-dimensional abelian.

If $k$ is odd and $\vec{d}$ is odd-symmetric  with $d_1=d_{k}=1$,
then Proposition \ref{ida} implies $r_{1,k}=0$ and hence the
nilpotency degree of $\n(C)$ is less than or equal to $k-2$.
If $\vec{d}\ne (1,\dots,1)$, then
Proposition \ref{crucial} part \eqref{vuelta0} applied to the
sequence $(d_1,\dots,d_{k-1})$ implies that $r_{1,k-1}\ne0$ and hence the
nilpotency degree of $\n(C)$ is equal to $k-2$.

Conversely, if $r_{1,k}=0$ then Proposition \ref{crucial} part
 \eqref{vuelta0} ensures that if $\vec d\neq (1,\dots,1)$ then $k$ is odd and $\vec d$ is odd symmetric with $d_1=d_k=1$.

As for the last statement, one implication follows from  Proposition \ref{crucial} part (4), while the reverse implication is a consequence
Proposition \ref{crucial} part (2) together with Propositions \ref{ida}, \ref{prop.r012} and \ref{prop.dominante}. 
\end{proof}

\begin{cor}
 If $l<k$ and $r_{i,l+i}=0$ for $i=1,\dots,k-l$,
 then $\vec{d}=(1,\dots,1)$.
\end{cor}
\begin{proof}
By hypothesis, all the sequences
\[
(d_1,\dots,d_{l+1}),(d_2,\dots,d_{l+2}), \dots ,
(d_{k-l},\dots,d_{k})
\]
fall in  parts
\eqref{sym4} and \eqref{vuelta0}
of Proposition \ref{crucial}.
\end{proof}

\subsection{The Lie algebra \texorpdfstring{$\n(S)$}{n(S)} for an arbitrary sequence \texorpdfstring{$S$}{S}}
\label{subsec.arbitrary_S}

In this subsection, the involution
$\phi$, 
defined in \S\ref{subsec.gradings},  plays a important role. 
Thus, for a matrix $A$, set $A=A_{+}+A_{-}$,  where
$\phi(A_{\pm})=\pm A_{\pm}$. 
Also, if $A\ne0$, let 
\[
\text{ddeg}(A)=\min\{j-i:A_{i,j}\ne0\}.
\]

Let $\vec{d}=(d_1,\dots,d_{k})$, $k\ge 2$,
and let $S=(S(1),\dots,S(k-1))$ be as in  \eqref{eq.Condition_S}. 
By Proposition \ref{prop.struture}, if $\n(C)$ is $(k-1)$-steps nilpotent (maximal possible nilpotency degree) then so it is $\n(S)$.
On the other hand, it might happen that
$\n(S)$ is still is $(k-1)$-steps nilpotent
when the nilpotency degree of $\n(C)$ is less then $k-1$.
According to Theorem \ref{thm.main}, if the
the nilpotency degree of $\n(C)$ is less then $k-1$, then it is
either equal to 1 or equal to $k-2$.
Theorem \ref{thm.main} says that it is equal to 1 if only if  $\vec{d}=(1,\dots,1)$ (in which case $S=C$);
and it is equal to $k-2$ if only if
 $k$ is odd and $\vec{d}$ is odd-symmetric  with $d_1=d_{k}=1$, 
 in which case the nilpotency degree of $\n(S)$ is $k-2$ or $k-1$ by Proposition \ref{prop.struture}.
Theorem \ref{thm.main2} below elucidates this dichotomy. 
We begin with the following technical result.

\begin{lemma}\label{lemma.Yr}
Given $k\ge 3$ odd, let $\vec{d}=(d_1,\dots,d_{k})$ be odd-symmetric,
set $d=|\vec{d}|$,  and let $S=(S(1),\dots,S(k-1))$ be as in  \eqref{eq.Condition_S}.
Let $E=E(S)$ and let us decompose $E$ as $E=E_{+}+E_{-}$,  where
$\phi(E_{\pm})=\pm E_{\pm}$.

If $\text{ddeg}\,(E_-)> r $ (notice that $1\le r< d_1+d_2-1)$, then
there exists 
a $\phi$-invariant $Y\in\n(S)$ with $\text{ddeg}\,(Y)=d-d_1-r$ 
such that 
\[
\begin{array}{rl}
Y_{1,d-d_1-r+1}=1, & \text{if $r\le d_2$;} \\[1mm]
Y_{1+r-d_2,d-d_1-d_2+1}=1, & \text{if $r\ge d_2$.} 
\end{array}
\]
\end{lemma}

\begin{proof}
Let us denote $S_+$ and $S_-$ the sequences corresponding to 
$E_+$ and $E_-$ respectively. 

We first assume $r\le d_2=d_{k-1}$. In this case, the 
sequence 
\[
(d_1,\dots,d_{k-2},d_{k-1}+1-r)
\]
is not odd-symmetric and hence
Theorem \ref{thm.main} 
ensures the existence of a matrix $X\in\n(C)$ with  
$\text{ddeg}\,(X)=d-d_1-r$ 
such that
$X_{1,d-d_1+1-r}=1$ (recall that $d_1=d_k$).
This $X$ is clearly  $\phi$-invariant. 

Now we proceed as in the proof of Proposition \ref{algraduada}
to obtain a matrix $\tilde X\in\n(S_+)$ with  
$\text{ddeg}\,(\tilde X)=d-d_1-r$ 
such that
$\tilde X_{1,d-d_1+1-r}=1$.
Again, it  is clear that $\tilde X$ is $\phi$-invariant. 

Next we proceed again as in the proof of Proposition \ref{algraduada}
to obtain 
$Y\in\n(S)$ with 
 $\text{ddeg}\,(Y)=d-d_1-r$ 
 such that
\[
\tilde X= Y + M_1
\]
with $M_1=0$ or (by the hypothesis 
$\text{ddeg}\,(E_-)> r $)  
$ \text{ddeg}\,(M_1) > \text{ddeg}\,(X)+ r$. 
But this impĺies that $M_1=0$ since otherwise its  block-degree would be $k-1$, this 
is a contradiction since the  block-degree of 
$\tilde X$ and $Y$ is $k-2$.
Therefore,  $Y$ is $\phi$-invariant and 
$Y_{1,d-d_1+1-r}=1$.

Now assume $r\ge d_2=d_{k-1}$. In this case, the argument is completely analogous, except for the use of the sequence 
\[
(d_1+d_2-r,d_2\dots,d_{k-2},1)
\]
which is, again, not odd-symmetric.
\end{proof}

\begin{prop}\label{eluci}
Given $k\ge 3$ odd, let $\vec{d}=(d_1,\dots,d_{k})$ be odd-symmetric
(not necessarily $d_1=1$),  set $d=|\vec{d}|$,  and let $S=(S(1),\dots,S(k-1))$ be as in  \eqref{eq.Condition_S}.
Assume that $S$ is weakly normalized.
Then  $A_{1,d}=0$ for all $A\in\n(S)$
if and only if $S$ is $\phi$-invariant.
\end{prop}

\begin{proof} We know, from Proposition \ref{ida} that
$A_{1,d}=0$ for all $A\in\n(S)$
if $S$ is $\phi$-invariant.

Conversely,
set $D=D(0,0)$, $E=E(S)$.
We  decompose $E=E_{+}+E_{-}$ and we want to prove that $E_-=0$.
It is clear that, if $E_-\neq0$, then ${\rm ddeg}(E_-)>{\rm ddeg}(E_+)=1$.

We will prove by induction on $k$ (which will always be odd) that $E_{-}=0$.

\medskip

We first point out that,
since $A_{1,d}=0$ for all $A\in\n(S)$ it follows that 
\begin{equation}\label{eq.llc1}
A_{d_1,d+1-d_1}=0\text{ for all $A\in\n(S)$,}
\end{equation}
and 
\begin{equation}\label{eq.llc2}
A_{d_1+d_2,d+1-d_1-d_2}=0\text{ for all $A\in\n(S)$.}
\end{equation}
Indeed, 
if $A_{d_1,d+1-d_1}\ne0$
then $(\text{ad}(D)^{2d_1-2}(A))_{1,d}\ne0$; and 
if $A_{d_1+d_2,d+1-d_1-d_2}\ne0$
then $[(\text{ad}(D)^{d_1+d_2-1}(E),[(\text{ad}(D)^{d_1+d_2-1}(E),A]])_{1,d}\ne0$;

\

{\sc Case $k=3$}. 
If $d_1=1$, then
\[
E=
\left(
\begin{array}{c|cccc|c}
        &1 & a_2 &\cdots & a_{d_2}   &         \\
	\hline
	   &   &     &       &         & a_{d_2} \\
	  &   &      &        &       & \vdots   \\
	   &   &     &       &         & a_{2} \\
	   &   &     &       &         & 1  \\
	\hline
	      &      &        &       &
	\end{array}
\right)
\]
and we must prove $(E_-)_{1,j}=0$ for $j=2,\dots,d_2+1$, that is
$a_{2j}=0$ for $j=1,\dots, \frac{d_2-1}2$ (recall that $d_2$ is odd).
If we know that $a_{2j}=0$ for $j=1,\dots, j_0-1<\frac{d_2-1}2$ then
\[
\Big([E,\text{ad}(D)^{d_2-2j_0}(E)]\Big)_{1,d}=4a_{2j_0},
\]
and hence we obtain $a_{2j_0}=0$. This completes the case $d_1=1$.

Assume $d_1>1$.
The following argument will be used later again in a very similar way,
therefore we label it as Argument (*).

Since $A_{d_1,d+1-d_1}=0$ for all $A\in\n(S)$ (see \eqref{eq.llc1}), the case $d_1=1$
implies
\begin{equation}\label{eq.induction_d1=1}
(E_-)_{d_1,j}=0,\quad\text{for }j=d_1+1,\dots,d_1+d_2.
\end{equation}
In order to prove that $E_-=0$ we will prove, by induction on $t$
(up to $d_1+d_2-1$), 
that
\begin{equation}\label{eq.to_be_proved}
(E_-)_{i,j}=0,\quad\text{for }i=1,\dots,d_1;\; j=d_1+1,\dots,d_1+d_2;
\text{ with } t=j-i.
\end{equation}
Since we already have $(E_-)_{d_1,d_1+1}=0$,
the case $t=1$ is done.
Now assume that we have proved \eqref{eq.to_be_proved}
for  $j-i\le t$ and let us prove it for $j-i=t+1$.

Since \eqref{eq.to_be_proved} is true for  $j-i\le  t$,
it follows that
\[
\text{ad}(D)^{d_1+d_2-1-r}(E_{-})=0,\text{ for $r=1,\dots,t$},
\]
and hence
\[
Y_r=\text{ad}(D)^{d_1+d_2-1-r}(E)\in\n(S)
\]
is $\phi$-invariant for $r=1,\dots,t$.
Now, if there is a $j$, with $i=j-t-1<d_1$, so that
$(E_-)_{i,j}\ne 0$, let $j_0$ be the first such one and let 
$i_0=j_0-t-1$ be the corresponding row, then
\begin{align*}
\Big(\text{ad}(D)^{i_0-1}([Y_{j_0-d_1},E])\Big)_{1,d} &=
\Big(\text{ad}(D)^{i_0-1}([Y_{j_0-d_1},E_-])\Big)_{1,d} \\
& = - 2(Y_{j_0-d_1})_{j_0,d}(E_-)_{i_0,j_0},
\end{align*}
a contradiction (the first equality is a consequence of 
the odd-symmetry of $\vec d$).
Since \eqref{eq.induction_d1=1} takes care of the row $i=d_1$,
this completes the case $k=3$.

\medskip

{\sc Case $k\ge 5$}. Let $k\ge 5$ be an arbitrary odd integer and assume the result
proved for $k'$ odd and $k'\le k-2$.

First the case $d_1=1$. 
The induction hypothesis and \eqref{eq.llc2} 
imply that
\[
 (S(2)',S(3),\dots,S(k-3),S(k-2)')
\]
is $\phi$-invariant, here $S(2)'$ is the last row of $S(2)$ and 
$S(k-2)'$ is the first column of $S(k-2)$.
Therefore, if $d_2=1$ we are done. 
Otherwise, assume $d_2=d_{k-1}>1$. 

Now we proceed as in Argument (*), where the rol played by 
$d_1$ is now played by $d_2$, as follows.
We know that 
\[
(E_-)_{1+d_2,j}=0,\quad\text{for }j=1+d_2+1,\dots,1+d_2+d_3.
\]
In order to prove that $E_-=0$ we will show, by induction on $t$
(up to $d_2+d_3-1$), 
that
\begin{equation}\label{eq.to_be_proved11}
(E_-)_{i,j}=0,\quad\text{for }i=2,\dots,1+d_2;\; j=d_2+2,\dots,d_2+d_3+1;
\text{ with } t=j-i.
\end{equation}
Since we already have $(E_-)_{1+d_2,d_2+2}=0$,
the case $t=1$ is done.
Now assume that we have proved \eqref{eq.to_be_proved11}
for  $j-i\le t$ and let us prove it for $j-i=t+1$.

Since \eqref{eq.to_be_proved11} is true for  $j-i\le  t$,
it follows that 
\[
\text{ad}(D)^{d_2+d_3-1-r}(E_{-})=0,\text{ for $r=1,\dots,t$},
\]
and hence
\[
Y_r=\text{ad}(D)^{d_2+d_3-1-r}(E)\in\n(S)
\]
is $\phi$-invariant for $r=1,\dots,t$.
Now, if there is a $j$, with $i=j-t-1<1+d_2$, so that
$(E_-)_{i,j}\ne 0$, let $j_0\ge d_2+1$ 
be the first such one and let 
$i_0=j_0-t-1$ be the corresponding row, then
\begin{align*}
\Big(\text{ad}(D)^{i_0-2}([Y_{j_0-d_2-1},E])\Big)_{2,d-1} &=
\Big(\text{ad}(D)^{i_0-2}([Y_{j_0-d_2-1},E_-])\Big)_{2,d-1} \\
& =- 2{(Y_{j_0-d_2-1})_{j_0,d-1}}(E_-)_{i_0,j_0}.
\end{align*}

Now, the $\phi$-invariance of $E$ up to degree $t$ and consequently
the $\phi$-invariance of 
\[
\text{ad}(D)^{i_0-2}([Y_{j_0-d_2-1},E])
\]
 up to degree $d-3$  
imply
\[
\frac{1}{c}[E,[E,\text{ad}(D)^{i_0-2}([Y_{j_0-d_2-1},E])]]_{1,d}=
 (Y_{j_0-d_2-1})_{j_0,d-1}(E_-)_{i_0,j_0}\ne 0,
\]
($c$ is a non-zero scalar) a contradiction.

\medskip

Finally we consider the case $d_1>1$.
Again, since  $A_{d_1,d-d_1+1}=0$
for all $A\in\n(S)$, the case $d_1=1$
implies that $(S(2),\dots,S(k-2))$
is $\phi$-invariant and
\[
E_{d_1,j}=(-1)^{j-d_1-1}E_{d+1-j,d+1-d_1}\quad\text{for } j=d_1+1,\dots,d_1+d_2.
\]
Therefore we only need to show that
\[
E_{i,j}=(-1)^{j-i-1}E_{d+1-j,d+1-i}\quad\text{for }i=1,\dots,d_1-1\text{ and }j=d_1+1,\dots,d_1+d_2
\]
which is equivalent to prove that
\begin{equation}\label{eq.to_be_proved1}
(E_-)_{i,j}=0,\quad\text{for }i=1,\dots,d_1-1\text{ and }j=d_1+1,\dots,d_1+d_2.
\end{equation}
Now we use  Argument (*) again. 
We will prove \eqref{eq.to_be_proved1} by induction
 on $t=j-i\le d_1+d_2-1$.
By the weak normalization of $(S(1),\dots, S(k-1))$,
the case $t=1$ is complete.

Let $Y_r$, $r=1,\dots,d_2$, be the $\phi$-invariant matrices provided by 
Lemma \ref{lemma.Yr}.
Now, if there is a $j$, with $i=j-t-1<d_1$, so that
$(E_-)_{i,j}\ne 0$, let $j_0$ be the first such one and let 
$i_0=j_0-t-1$ be the corresponding row, then
\begin{align*}
\Big(\text{ad}(D)^{i_0-1}([Y_{j_0-d_1},E])\Big)_{1,d} &=
\Big(\text{ad}(D)^{i_0-1}([Y_{j_0-d_1},E_-])\Big)_{1,d} \\
& = \pm 2(E_-)_{i_0,j_0}\ne0,
\end{align*}
a contradiction, and this completes the proof.
\end{proof}

We may now state one of the main results of the paper. 

\begin{theorem}\label{thm.main2} Let $k\ge 2$, $\vec{d}=(d_1,\dots,d_{k})$,
$S=(S(1),\dots,S(k-1))$ an arbitrary sequence satisfying \eqref{eq.Condition_S},
and $T=(T(1),\dots,T(k-1))$ the only normalized sequence such that $E(T)$ is $G(\vec d)$-conjugate to $E(S)$,
as ensured by Proposition \ref{prop.normalized_h}. Then the nilpotency degree of $\n(S)$ is $k-1$ except when
 \begin{enumerate}
\item $\vec{d}=(1,\dots,1)$, in which case $\n(S)$ is 1-dimensional abelian.
\item $k$ is odd, $\vec{d}$ is odd-symmetric with $d_1=d_{k}=1$ and $T$ is $\phi$-invariant,
in which case, if $\vec{d}\ne (1,\dots,1)$,
the nilpotency degree of $\n(S)$ is $k-2$.
\end{enumerate}
 \end{theorem}

 \begin{proof}
This is an immediate consequence of the comments made at the beginning of the subsection and the previous proposition.
 \end{proof}

\begin{remark}\label{rml.char_p}
Theorem \ref{thm.main2} fails spectacularly when $\F$ has prime characteristic $p$. Indeed, let $k$ be arbitrary and
take $d_i=p$ for all $i=1,\dots,k$, that is
$\vec d=(p,\dots,p)$.
Let $C=(C(1),\dots,C(k-1))$ be the canonical sequence,
this means that, for all $i=1,\dots,k-1$, we have $C(i)=C_0$ with
\[
C_0=\left(
               \begin{array}{cccc}
                 0 &  0 & \dots & 0 \\
                 \vdots & \vdots  &  & \vdots  \\
                 0 & 0 & \dots & 0\\
                 1 & 0 & \dots  & 0
               \end{array}
             \right)\in M_{p\times p}.
\]
We claim that that $\n(C)$ is abelian (instead of having
nilpotency degree $k-1$).
 Since $(J^p(0))^p=0$, the binomial
 theorem implies $\big(\ad_{\gl(p)}\, J^p(0)\big)^{p}=0$.
 Therefore, $\big(\ad_{\gl(kp)}\, D(0,0)\big)^{p}=0$
 (we notice that in $\text{char}\,\F=0$ we would have had
 $\big(\ad_{\gl(kp)}\, D(0,0)\big)^{2(p-1)}\ne 0$).

 Therefore, the generators of $\n(C)$ are
 (see \eqref{eq.generators})
 \[
 V_r=\big(\ad_{\gl(kp)}\, D(0,0)\big)^r(E(C)),\quad r=0,\dots,p-1.
 \]
It is clear that, for all $i=1,\dots,k-1$, the block
$p_{i,i+1}(V_r)$ is equal to
\[
W_r=\big(\ad_{\gl(p)}\, J^p(0)\big)^r(C_0).
\]
Now, $W_0=C_0$ commutes with
all strictly lower triangular matrices as well as with the diagonal matrix $W_{p-1}$, whose
 only non-zero
 entries are equal to 1 and appear in positions $(1,1)$ and $(p,p)$.
Thus $[W_r,W_0]=0$ for all $r$, whence $[W_r,W_s]=0$ for all $0\le r,s<p$.
It follows that $\n(C)$ is abelian.

Another curious example is $\n(C)$, with
$\vec d=(2,3,2,3,2)$ in $\text{char}\,\F=2$.
Here, a basis of $\n(C)$ is

\[\tiny
\arraycolsep=1.1pt\def\arraystretch{.9}
\left(
\begin{array}{cc|ccc|cc|ccc|cc}\hspace{3mm}
     &  & 0 & 0 & 0 &  &  &   &   & & &  \\
     &  & 1 & 0 & 0 &  &  &   &   & & &  \\
	\hline
     &  &  &  &  & 0 & 0 &   &   & & &  \\
     &  & &  &  & 0 & 0 &   &   & & &  \\
     &  &  &  &  & 1 & 0 &   &   & & &  \\
	\hline
     &  &  &  &  &  &  & 0 & 0  & 0 & &  \\
     &  &  &  &  &  &  & 1 & 0  & 0 & &  \\
	\hline
      &  &  &  &  &  &  &   &   & & 0 &0  \\
     &  & &  &  &  &  &   &   & & 0 & 0 \\
     &  &  &  &  &  &  &   &   & & 1&0  \\
	\hline
     &  &  &  &  &  &  &   &   & & &  \\
     &  &  &  &  &  &  &   &   & & &  \\
	\end{array}
\right)
\left(
\begin{array}{cc|ccc|cc|ccc|cc}\hspace{3mm}
     &  & 1 & 0 & 0 &  &  &   &   & & &  \\
     &  & 0 & 1 & 0 &  &  &   &   & & &  \\
	\hline
     &  &  &  &  & 0 & 0 &   &   & & &  \\
     &  & &  &  & 1 & 0 &   &   & & &  \\
     &  &  &  &  & 0 & 1 &   &   & & &  \\
	\hline
     &  &  &  &  &  &  & 1 & 0  & 0 & &  \\
     &  &  &  &  &  &  & 0 & 1  & 0 & &  \\
	\hline
      &  &  &  &  &  &  &   &   & & 0 & 0  \\
     &  & &  &  &  &  &   &     & & 1 & 0 \\
     &  &  &  &  &  &  &   &    & & 0 & 1  \\
	\hline
     &  &  &  &  &  &  &   &   & & &  \\
     &  &  &  &  &  &  &   &   & & &  \\
	\end{array}
\right)
\left(
\begin{array}{cc|ccc|cc|ccc|cc}\hspace{3mm}
     &  & 0 & 0 & 0 &  &  &   &   & & &  \\
     &  & 0 & 0 & 1 &  &  &   &   & & &  \\
	\hline
     &  &  &  &  & 1 & 0 &   &   & & &  \\
     &  & &  &  & 0 & 0 &   &   & & &  \\
     &  &  &  &  & 0 & 0 &   &   & & &  \\
	\hline
     &  &  &  &  &  &  & 0 & 0  & 0 & &  \\
     &  &  &  &  &  &  & 0 & 0  & 1 & &  \\
	\hline
      &  &  &  &  &  &  &   &   & & 1 & 0  \\
     &  & &  &  &  &  &   &     & & 0 & 0 \\
     &  &  &  &  &  &  &   &    & & 0 & 0  \\
	\hline
     &  &  &  &  &  &  &   &   & & &  \\
     &  &  &  &  &  &  &   &   & & &  \\
	\end{array}
\right)
\left(
\begin{array}{cc|ccc|cc|ccc|cc} \hspace{3mm}
     &  & 0 & 0 & 1 &  &  &   &   & & &  \\
     &  & 0 & 0 & 0 &  &  &   &   & & &  \\
	\hline
     &  &  &  &  & 0 & 1 &   &   & & &  \\
     &  & &  &  & 0 & 0 &   &   & & &  \\
     &  &  &  &  & 0 & 0 &   &   & & &  \\
	\hline
     &  &  &  &  &  &  & 0 & 0  & 1 & &  \\
     &  &  &  &  &  &  & 0 & 0  & 0 & &  \\
	\hline
      &  &  &  &  &  &  &   &   & & 0 & 1  \\
     &  & &  &  &  &  &   &     & & 0 & 0 \\
     &  &  &  &  &  &  &   &    & & 0 & 0  \\
	\hline
     &  &  &  &  &  &  &   &   & & &  \\
     &  &  &  &  &  &  &   &   & & &  \\
	\end{array}
\right)
\]
\[\tiny
\arraycolsep=1.1pt\def\arraystretch{.9}
\left(
\begin{array}{cc|ccc|cc|ccc|cc} \hspace{3mm}
     &  & \hspace{5mm}  &  &  & 1 & 0 &   &   & & &  \\
     &  &  &  &  & 0 & 1 &   &   & & &  \\
	\hline
      &  &  &  &  &  &  & 1 & 0 &0 & &  \\
     &  & &  &  &  &    & 0 & 0 &0 & &  \\
     &  &  &  &  &  &   & 0 & 0 &1 & &  \\
	\hline
     &  &  &  &  &  &  &   &   & & 1 & 0 \\
     &  &  &  &  &  &  &   &   & & 0 & 1  \\
	\hline
      &  &  &  &  &  &  &   &   & & &  \\
     &  & &  &  &  &  &   &   & & &  \\
     &  &  &  &  &  &  &   &   & & &  \\
	\hline
     &  &  &  &  &  &  &   &   & & &  \\
     &  &  &  &  &  &  &   &   & & &  \\
	\end{array}
\right)
\left(
\begin{array}{cc|ccc|cc|ccc|cc} \hspace{3mm}
     &  & \hspace{5mm} &  &  & 0 & 0 &   &   & & &  \\
     &  &  &  &  & 0 & 0 &   &   & & &  \\
	\hline
      &  &  &  &  &  &  & 0 & 1 &0 & &  \\
     &  & &  &  &  &    & 0 & 0 &1 & &  \\
     &  &  &  &  &  &   & 0 & 0 &0 & &  \\
	\hline
     &  &  &  &  &  &  &   &   & & 0 & 0 \\
     &  &  &  &  &  &  &   &   & & 0 & 0  \\
	\hline
      &  &  &  &  &  &  &   &   & & &  \\
     &  & &  &  &  &  &   &   & & &  \\
     &  &  &  &  &  &  &   &   & & &  \\
	\hline
     &  &  &  &  &  &  &   &   & & &  \\
     &  &  &  &  &  &  &   &   & & &  \\
	\end{array}
\right)
\left(
\begin{array}{cc|ccc|cc|ccc|cc} \hspace{3mm}
     &  & \hspace{5mm} &  &  &  & \hspace{3mm} & 0  &  0 & 0& &  \\
     &  &  &  &  &  &  & 0  & 1  & 0& &  \\
	\hline
      &  &  &  &  &  &  &   &   & & 0 &  0 \\
     &  & &  &  &  &  &   &   & & 1 & 0 \\
     &  &  &  &  &  &  &   &   & & 0 & 0 \\
	\hline
     &  &  &  &  &  &  &   &   & & &  \\
     &  &  &  &  &  &  &   &   & & &  \\
	\hline
      &  &  &  &  &  &  &   &   & & &  \\
     &  & &  &  &  &  &   &   & & &  \\
     &  &  &  &  &  &  &   &   & & &  \\
	\hline
     &  &  &  &  &  &  &   &   & & &  \\
     &  &  &  &  &  &  &   &   & & &  \\
	\end{array}
\right)
\left(
\begin{array}{cc|ccc|cc|ccc|cc} \hspace{3mm}
     &  & \hspace{5mm} &  &  &  & \hspace{3mm} & 0  &  1 & 0& &  \\
     &  &  &  &  &  &  & 0  & 0  & 1& &  \\
	\hline
      &  &  &  &  &  &  &   &   & & 1 &  0 \\
     &  & &  &  &  &  &   &   & & 0 & 1 \\
     &  &  &  &  &  &  &   &   & & 0 & 0 \\
	\hline
     &  &  &  &  &  &  &   &   & & &  \\
     &  &  &  &  &  &  &   &   & & &  \\
	\hline
      &  &  &  &  &  &  &   &   & & &  \\
     &  & &  &  &  &  &   &   & & &  \\
     &  &  &  &  &  &  &   &   & & &  \\
	\hline
     &  &  &  &  &  &  &   &   & & &  \\
     &  &  &  &  &  &  &   &   & & &  \\
	\end{array}
\right)
\]
and thus the nilpotency degree of $\n(C)$ is 3.

\begin{ques}\label{ques1}
For which $p\ne0$, $k$ and $1\le\ell<k$, 
there exists
$\vec d=(d_1,\dots,d_k)$ such that the nilpotency degree of $\n(C)$ is $\ell$?
\end{ques}
\end{remark}

\subsection{In which cases is
\texorpdfstring{$\n(S)$}{n(S)} a free
\texorpdfstring{$N$}{N}-step nilpotent Lie algebra?}
Assume that $S$ is weakly normalized.
We already know that if $\vec{d}=(1,\dots,1)$,
then  $\n(S)$ is 1-dimensional abelian.
It is also ablelian if $\vec{d}=(1,\rho,1)$ with $\rho$ odd
and a $\phi$-invariant $S=(S(1),S(2))$,
in this case of dimension $\rho$.
It follows from Theorem \ref{thm.main2} that
these are the only two cases where $\n(S)$ is abelian.
We now determine for which other sequences $\vec{d}=(d_1,\dots,d_{k})$
and  $S=(S(1),\dots,S(k-1))$,
the Lie algebra $\n=\n(S)$ is a free $N$-step nilpotent Lie algebra, $N\ge 2$.

We know (see Proposition \ref{prop.struture}) that
$\n$ is generated by the linearly independent matrices
\[
\big(\text{ad}(D)\big)^{j-1}(E),\quad j=1,\dots,\rho
\]
with $D=D(0,0)$, $E=E(S)$ and  $\rho=\max\{d_i+d_{i+1}-1:1\le i\le k-1\}$.
We assume now $\vec{d}\ne(1,\dots,1)$ and hence $\rho\ge2$.

Witt's formula tell us
that the dimension of the center of the free $N$-step
nilpotent Lie algebra in $\rho$ generators, say ${\mathcal F}_{\rho, N}$, is
\[
\frac{1}{N}\sum_{s\vert N}\mu(s)\rho^{N/s},
\]
where $\mu$ is the M\"obius function, and it is clear that 
\[
\frac{1}{N}\sum_{s\vert N}\mu(s)\rho^{\frac{N}{s}}\ge
\frac{1}{N}\left(\rho^N-\rho\,\frac{\rho^{\frac{N}{q}}-1}{\rho-1}\right)
=
\frac{1}{N(\rho-1)}\left(\rho^{N+1}-\rho^N-\rho^{\frac{N}{q}+1}+\rho\right)
\]
with $q$ the first prime dividing $N$.

We first consider the case when $k$ is odd, $\vec{d}$ is odd-symmetric with $d_1=d_{k}=1$, and $S$ is $\phi$-invariant.
In this case, by Theorem \ref{thm.main2}, we have that $N=k-2$ is odd and
the center of
$\n$ is contained $p_{1,k-1}(\n)\oplus p_{2,k}(\n)$.
Thus we must have
\[
\frac{1}{N(\rho-1)}\left(\rho^{N+1}-\rho^N-\rho^{\frac{N}{3}+1}+\rho\right)
\le 2d_2\le 2\rho
\]
that is
\[
\rho^{N+1}-\rho^N-\rho^{\frac{N}{3}+1}+\rho
\le 2N \rho (\rho-1)
\]
which is possible only if $(\rho,N)=(2,3)$.
This yields $\vec{d}=(1,2,1,2,1)$.
We claim that in this case $\n\simeq {\mathcal F}_{2, 3}$. Indeed,
Let $\mathfrak j$ be the kernel of the obvious epimorphism ${\mathcal F}_{2, 3}\to\n$ and let $\mathfrak z$ be the center of ${\mathcal F}_{2, 3}$.
Suppose, if possible, that $\mathfrak j\neq (0)$. Since ${\mathcal F}_{2, 3}$ is nilpotent, this implies $\mathfrak z\cap \mathfrak j\neq (0)$.
But ${\mathcal F}_{2, 3}^2=\mathfrak z$ and ${\mathcal F}_{2, 3}^2$ maps onto $\n^2$, so $\dim(\n^2)<\dim {\mathcal F}_{2, 3}^2=\dim(\mathfrak z)=2$.
However, we see by direct computation that $\dim(\n^2)=2$. This contradiction shows that $\n\simeq {\mathcal F}_{2, 3}$.

Now we consider the general case, in which $N=k-1$ and the center of
$\n$ is contained $p_{1,k}(\n)$.
The subcase $N=2$, that is $k=3$, has been done in \cite{CGS1} and the result is
that $\n$ is free 2-step nilpotent in $\rho$ generators if and only if
\[
\vec{d}\in \{
(\rho, 1, \rho), \;
(\rho - 1, 2, \rho - 1), \;
(\rho, 1, \rho - 1), \;
(\rho - 1, 1, \rho)\}.
\]
For the subcase $N>2$, we must have
\[
\frac{1}{N(\rho-1)}\left(\rho^{N+1}-\rho^N-\rho^{\frac{N}{2}+1}+\rho\right)
\le d_1d_k\le \rho^2
\]
that is
\[
\rho^{N+1}-\rho^N-\rho^{\frac{N}{2}+1}+\rho
\le N \rho^2 (\rho-1)
\]
which is possible only if $(\rho,N)=(2,3),\;(3,3),\;(2,4)$.

If $(\rho,N)=(2,3)$ then
\[
\vec{d}\in \{
(2,1,1,2),\, (2,1,2,1),\,  (1,2,1,2)
\}.
\]
 In each of these cases we have $\n\simeq {\mathcal F}_{2, 3}$. Indeed, arguing as above, it suffices to show
that $\dim(\n^2)=2$ in all cases.

 We next claim that if  $(\rho,N)=(3,3)$ then $\n\not\simeq {\mathcal F}_{3, 3}$. Indeed, the center of ${\mathcal F}_{3, 3}$ has dimension 8.
Now $8\leq d_1d_4$ forces $\vec{d}=(3,1,1,3)$. However, in this case a direct computation shows that $\dim(\n^2)<8$.

Finally, the case $(\rho,N)=(2,4)$
is possible only if
\[
\vec{d}=(2,1,2,1,2).
\]
In this case we do have $\n\simeq {\mathcal F}_{2, 4}$. Indeed, it suffices to show
that $\dim(\n^3)=3$.

We summarize these results in the following theorem.

\begin{theorem}\label{thm.free} Let $k\ge 2$, $\vec{d}=(d_1,\dots,d_{k})$,
$S=(S(1),\dots,S(k-1))$ an arbitrary sequence satisfying \eqref{eq.Condition_S},
and $T=(T(1),\dots,T(k-1))$ the only normalized sequence such that $E(T)$ is $G(\vec d)$-conjugate to $E(S)$,
as ensured by Proposition \ref{prop.normalized_h}. Then the Lie algebra $\n(S)$ is free $N$-step nilpotent in $\rho$ generators
if and only if
$\vec{d}$ is as follows:
\begin{enumerate}
    \item  $\vec{d}=(1,\dots,1)$, here $(\rho,N)=(1,1)$.
    \item  $\vec{d}=(1,d,1)$  with $T$ $\phi$-invariant, here $(\rho,N)=(d,1)$.
    \item  $\vec{d}\in \{
(d, 1, d), \;
(d - 1, 2, d- 1), \;
(d, 1, d - 1), \;
(d - 1, 1, d)\}$, here $(\rho,N)=(d,2)$.
\item $\vec{d}\in \{(2,1,1,2),\, (2,1,2,1),\,  (1,2,1,2)\}$,
here $(\rho,N)=(2,3)$.
\item  $\vec{d}=(1,2,1,2,1)$ with $T$ $\phi$-invariant, here $(\rho,N)=(2,3)$.
\item  $\vec{d}=(2,1,2,1,2)$, here $(\rho,N)=(2,4)$.
\end{enumerate}
In all cases, $d\ge2$.
\end{theorem}

\section{Uniserial representations of the Lie algebras
\texorpdfstring{$\g_{n,\lambda}$}{gn,l} and \texorpdfstring{$\g_{n,\lambda,\ell}$}{gn,l,l}}\label{sec.uniserials}

\subsection{The Lie algebras \texorpdfstring{$\g_{n,\lambda}$}{gn,l} and \texorpdfstring{$\g_{n,\lambda,\ell}$}{gn,l,l}}\label{subsec.def_Lie_alg}

Let $V$ be a vector space of dimension $n\geq 2$ and let
$\L(V)$ be the free Lie algebra associated to $V$
(or the free Lie algebra on $n$ generators).
For $\ell \ge 1$,  let
\[
\N_\ell(V)=\L(V)/\L(V)^\ell
\]
be
the free $\ell$-step nilpotent Lie algebra associated to $V$.

Let $x\in\text{End}(V)$ the linear map acting
on $V$ via a single Jordan block $J_n(\la)$.
In particular $V$ has a basis
$\{v_0,\dots,v_{n-1}\}$
such that
\begin{equation}\label{eq.00}
 ( x - \lambda 1_V)^{k} v_0 =
\begin{cases}
v_k,&\text{if $0 \le k < n$;} \\
0,&\text{if $k=n$.}
\end{cases}
 \end{equation}
 We extend the action of $x$ on $V$ to
 $\L(V)$ so that $x$ becomes a Lie algebra derivation.
 This action preserves $\L(V)^\ell$
and thus $x$ also acts by derivations  on $\N_\ell(V)$.
Let
\[
\g_{n,\lambda}=\langle x\rangle\ltimes \L(V)\quad\text{ and }\quad
\g_{n,\lambda,\ell}=\langle x\rangle\ltimes \N_\ell(V)
\]be the corresponding semidirect products.

\subsection{The uniserial representations \texorpdfstring{$R_{\vec d,\al,S}$}{Rd,a,S}}\label{subsec.R}
Recall that given a vector space $V$ of dimension $n$, $\g_{n,\lambda}=\langle x\rangle\ltimes \L(V)$
and $\g_{n,\lambda,\ell}=\langle x\rangle\ltimes \N_\ell(V)$
(see \S\ref{subsec.def_Lie_alg}).

Given a scalar $\al\in \F$, a sequence of positive integers
$\vec d =(d_1,\dots,d_{\ell+1})$
satisfying
\begin{equation}\label{eq.Condition_on_d_1}
 \max\{d_i+d_{i+1}:1\le i\le \ell\}\le n+1
\end{equation}
and a sequence  $S=(S(1),\dots,S({\ell}))$ as in  \eqref{eq.Condition_S},
we use  \eqref{eq.00}, \eqref{eq.Condition_on_d_1} and
the universal property of $\L(V)$
to define a representation
\[
R_{\vec d,\al,S}:\g_{n,\lambda}\to\gl(d),\quad d=|\vec d|,
\]
by  setting
\begin{align*}
R_{\vec d,\al,S}(x)& =D(\al,\la), \\
R_{\vec d,\al,S}(v_j)&=\big(\ad_{\gl(d)} D(\al,\la)-\la\big)^j (E(S)),\quad 0\leq j\leq n-1.
\end{align*}
We obtain
\begin{equation*}
R_{\vec d,\al,S}(\g_{n,\lambda}) =\h(\alpha, \lambda,S).
\end{equation*}
 It is clear that
$\L(V)^\ell \subset \ker(R_{\vec d,\al,S})$ and hence we
obtain  a representation of the truncated Lie algebra
\[
\bar R_{\vec d,\al,S}:\g_{n,\lambda,\ell}\to\gl(d).
\]
Since, for all $i=1,\dots,d-1$, either
$R(x)_{i,i+1}\ne0$ or
$R(v_0)_{i,i+1}\ne0$,
it follows that
$R_{\vec d,\al,S}$ and
$\bar R_{\vec d,\al,S}$ are
uniserial representations of $\g_{n,\lambda}$ and
$\g_{n,\lambda,\ell}$ respectively.

\begin{definition}
If the sequence $S$ is normalized, we say that
$R_{\vec d,\al,S}$ and $\bar R_{\vec d,\al,S}$
are \emph{normalized}.
\end{definition}

\begin{prop}\label{prop.normal2} Assume $\lambda\ne0$.
The normalized representations
$R_{\vec d,\al,S}$ (resp. $\bar R_{\vec d,\al,S}$) of
$\g_{n,\lambda}$ (resp. $\g_{n,\lambda,\ell}$)
are non-isomorphic to each other.
\end{prop}

\begin{proof} It is enough to consider the case for the
representations of
$\g_{n,\lambda}$.
Considering the eigenvalues of the image of $x$ as well as their multiplicities, the only possible isomorphisms are easily seen to be between $R_{\vec d,\al,S}$ and $R_{\vec d,\al,S'}$.
Assume that $R_{\vec d,\al,S}$ is isomorphic to $R_{\vec d,\al,S'}$.
Then there is $P\in\GL(|\vec d|)$ satisfying
\begin{equation}\label{eq.intertwining}
P R_{\vec d,\al,S}(y) P^{-1}=R_{\vec d,\al,S'}(y),
\quad\text{for all }y\in\g_{n,\lambda}.
\end{equation}
Considering $y=x$ in \eqref{eq.intertwining} we obtain that
$P$ must commute with $D(\al,\la)$, and hence
$P\in G(\vec d)$ (see Proposition \ref{prop.normalized_h}).
Finally, considering $y=v_0$ in \eqref{eq.intertwining},
it follows from Proposition \ref{prop.normalized_h}
that $S=S'$.
\end{proof}

\subsection{Classification of all uniserial \texorpdfstring{$\g_{n,\lambda}$}{gn,l}-modules}\label{sec:uniserials_free}

For the remainder of the paper we assume that $F$ is an algebraically closed field of characteristic 0.

In this section we classify all uniserial
(finite dimensional) representations of
$\g_{n,\lambda}=\langle x\rangle\ltimes \L(V)$, where $V$ is a vector space
of dimension $n$
over $\F$ 
on which $x$ acts via a single Jordan block $J_n(\la)$.
First we prove a proposition that provides information about
the structure of a uniserial representation of a wider class of Lie algebras.

\begin{prop}\label{prop.ensu} Let $\n$ be a solvable Lie algebra and let $x$ be a derivation of $\n$ such that $[\n,\n]$ has an $x$-invariant complement, say $\p$, in $\n$, and $x$ acts on $\p$ via a single Jordan block $J_n(\la)$, $\la\neq 0$. Let $v_0,\dots,v_{n-1}$ be a basis $\p$ such that
\begin{equation}\label{eq.x_v}
x(v_0)=\la v_0+v_1, x(v_1)=\la v_1+v_2,\dots, x(v_{n-1})=\la v_{n-1}.
\end{equation}
Set $\g=\langle x\rangle\ltimes \n$ and let
$T:\g\to\gl(U)$ be a uniserial representation of dimension $d$ such that $$\ker(T)\cap\p= 0.$$
Then there is a basis $\B$ of $U$, a unique scalar $\al\in\F$,
a unique sequence of positive integers
$\vec d=(d_1,\dots,d_{\ell+1})$, $\ell\ge1$, satisfying $|\vec d| =d$ and
\[
 \max\{d_i+d_{i+1}:1\le i\le \ell\}=n+1
\]
and a unique
normalized sequence $S=(S(1),\dots,S(\ell))$ of matrices
such that the matrix representation $R:\g\to\gl(d)$ associated to $T$ and $\B$ satisfies:
\begin{align}
\label{s1}
R(x)&=J^{d_1}(\alpha)\oplus J^{d_2}(\alpha-\la)\oplus\cdots\oplus J^{d_{\ell+1}}(\alpha-\ell\la), \\
\label{s2}
R(v_0)&=\left(
               \begin{array}{ccccc}
                 0 & S(1) & 0 & \dots & 0 \\
                 0 & 0 & S(2) & \dots & 0\\
                 \vdots & \vdots & \ddots & \ddots & \vdots\\
                 0 & 0 & \dots & \ddots & S(\ell) \\
                 0 & 0 & \dots & \dots & 0 \\
               \end{array}
             \right),
\end{align}
and every $R(y)$, $y\in\n$, is block strictly upper triangular relative to $\vec d$.
Moreover, if $\n^{k-1}$ is not contained in $\ker(T)$, then $\ell\geq k$.
\end{prop}

\begin{proof}
This proof follows the lines of the proof of \cite[Theorem 3.2]{CPS}.
It follows from  Lie's theorem that there is a basis
$\B =\{u_1 ,\dots, u_d \}$ of $U$
such that the corresponding matrix representation
$R:\g\to\gl(d)$ consists of upper triangular matrices.

Set
\[
D = R(x)\text{ and }E_k = R(v_k ),\; 0 \le k \le n - 1.
\]
Conjugating by an upper triangular matrix
(see \cite[Lemma 2.2]{CS} for the details) we may assume
that $D$ satisfies:
\begin{equation}\label{eq.1}
 D_{i,j} = 0\text{ whenever }D_{i,i} \ne D_{j,j}.
\end{equation}
Since $\lambda\ne0$ we have that the action of $x$ on $\p$
is invertible and hence $\p\subset [\g, \g]$.
This implies that
\begin{equation}\label{eq.0}
\begin{split}
E_k\text{ is strictly upper triangular for all }0 \le k \le n - 1, \\
\text{ and hence $R(v)_{i,i+1}=0$  for all $1 \le i < d$ and $v\in[\n,\n]$}.
\end{split}
\end{equation}

On the other hand we know, from \cite[Lemma 2.1]{CS},
 that for every $1 \le i \le d$ there is some
$y\in\g$ such that
\begin{equation}\label{eq.01}
 R(y)_{i,i+1}\ne 0.
 \end{equation}
 This, combined with
 \eqref{eq.0} and \eqref{eq.1}, imply that
\begin{equation}\label{eq.2}
 \text{if $D_{i,i}\ne D_{i+1,i+1}$ then $R(v)_{i,i+1}\ne 0$ for some $v\in\p$.}
 \end{equation}

\medskip

\noindent
\textsl{Step 1.} If $D_{i,i}\ne D_{i+1,i+1}$ then
$D_{i,i}- D_{i+1,i+1}=\lambda$ and  $(E_0)_{i,i+1}\ne 0$.

Indeed,  since $T$ is a representation,
it follows from \eqref{eq.x_v} that,
for  $1\le i < d$,
\begin{equation}\label{eq.3}
 (\ad_{\gl(d)} D - \lambda)^{k} E_0 =
\begin{cases}
E_k,&\text{if $0 \le k < n$;} \\
0,&\text{if $k=n$.}
\end{cases}
 \end{equation}
Since $D$ is upper triangular and $E_0$ is strictly upper triangular, this implies that, for  $1\le i < d$,
\begin{equation}\label{eq.4}
(D_{i,i} - D_{i+1,i+1}- \lambda)^k(E_0)_{i,i+1} =
\begin{cases}
(E_k)_{i,i+1},&\text{if $0 \le k < n$;} \\
0,&\text{if $k=n$.}
\end{cases}
 \end{equation}
Now, if  $D_{i,i}\ne D_{i+1,i+1}$ then it follows from
\eqref{eq.2} and \eqref{eq.4} that
 $(E_0)_{i,i+1}\ne 0$ and case $k=n$ in
\eqref{eq.4} implies
$D_{i,i}- D_{i+1,i+1}=\lambda$.

\medskip

\noindent
\textsl{Step 2.} For some integer $\ell\geq 0$, there is a unique
sequence $\vec d=(d_1,\dots,d_{\ell+1})$ of positive integers,
with  $d=|\vec d|$, such that
\[
D = D_1 \oplus \cdots \oplus D_{\ell+1},\quad D_i\in \gl(d_i),
\]
where each $D_i$ has scalar diagonal of scalar
$\alpha_i = \alpha - (i - 1)\lambda$ for some $\alpha\in\F$.

This follows at once from \eqref{eq.1} and Step 1, uniqueness
is a consequence of the arrangement of the eigenvalues of $D$.

\medskip

\noindent
\textsl{Step 3.} According to the  block structure of $\gl(d)$
given by $\vec d$,  $p_{r,r}(E_k)=0$ for all
$1\le r \le \ell+1$ and $0\le k\le n-1$.

Indeed, setting $U^j = \text{span}\{u_1 , \dots , u_j \}$ (each $U^j$ is a $\g$-submodule of $U$), we have to
show that the endomorphism induced by $E_k$, say $\bar E_k$, in
\[
\bar U^{r}=U^{d_1+\cdots+d_r}/U^{d_1+\cdots+d_{r-1}}
\]
 is zero.
On the one hand, the endomorphism induced by $\ad_{\gl(d)} D$
in $\gl(\bar U^{r})$ is nilpotent.
On the other hand, it follows from \eqref{eq.3} that $\bar E_k$
is a generalized eigenvector of eigenvalue
$\lambda$ of  the endomorphism induced by $\ad_{\gl(d)} D$.
Since $\lambda\ne0$ this is a contradiction.

\medskip

\noindent
\textsl{Step 4.} According to the  block structure of $\gl(d)$
given by $\vec d$, if $1\le i <j\le \ell+1$  and $j\ne i+1$, then
$p_{i,j}(E_k)=0$ for all $0\le k\le n-1$.

The proof of this uses the same argument used in the proof of Step 3.
The point is that $p_{i,j}(E_k)$ corresponds to an eigenvector
of eigenvalue $(j-i)\lambda$ of $\ad_{\gl(d)} D$ and, if $j-i\ne1$,
\eqref{eq.3} implies that $p_{i,j}(E_k)$ must be zero.

\medskip

\noindent
\textsl{Step 5.}
Let $\alpha$ as in Step 2. We may assume that $D$ is in Jordan form
 \[
D=J^{d_1}(\al)\oplus J^{d_2}(\al-\la)\oplus\cdots\oplus J^{d_{\ell+1}}(\al- \ell\la).
\]
Moreover, $\ell\ge 1$ and if
$\n^{k-1}$ is not contained in $\ker(T)$, then $\ell\geq k$.

Indeed, by \eqref{eq.01} and Step 3,
the first superdiagonal of every $D_i$
consists entirely of non-zero entries.
Thus, for each $1 \le i \le \ell+1$, there is $P_i \in GL(d_i )$
such that
\[
P_i D_iP_i^{-1} = J^{d_i}(\al-(i-1)\la).
\]
Set $P = P_1 \oplus\cdots\oplus P_{\ell+1}\in GL(d)$,
then $PDP^{-1}$ is as stated and and $PE_k P^{-1}$ is still strictly block upper triangular with
$p_{i,j}(PE_k P^{-1})=0$ if $1\le i \le j\le \ell+1$ and $j-i\ne 1$.
Since $\n^{k-1}$ is obtained by bracketing elements of $\p$,
it follows from Step 3 that, if  $\ell<k$, then
$\n^{k-1}\subset\ker(T)$.
In particular, since $\ker(T)\cap\p= 0$, we have $\ell\ge1$.

\medskip

\noindent
\textsl{Step 6.}
For all $1\le i \le \ell$, $d_i+d_{i+1}\leq n+1$ and
the equality holds for some $i$.

Indeed, from Step 1 we know that $(E_0)_{d_i,d_i+1}\ne 0$ for all $i$.
If $d_i+d_{i+1}>n+1$, for some $i$, it follows from the Clebsh-Gordan decomposition of the tensor product of irreducible representations of $\sl(2)$ that
$(\ad_{\gl(d)} D - \lambda 1_{\gl(d)})^{n} E_0 \ne 0$, contradicting \eqref{eq.3}
(for the details, see \cite[Proposition 2.2]{CPS}).
On the other hand,
if $d_i+d_{i+1}<n+1$ for all $i$ then Clebsh-Gordan implies that
$E_n=(\ad_{\gl(d)} D - \lambda 1_{\gl(d)})^{n-1}E_0 = 0$, which is impossible
since $\ker(T)\cap\p= 0$.

\medskip

\noindent
\textsl{Final Step.} We may assume $E_0=\tiny\left(
               \begin{matrix}
                 0 & S(1)  & \dots & 0 \\
                 \vdots  & \ddots & \ddots & \vdots\\
                 0  &  & \ddots & S(\ell) \\[1mm]
                 0  & 0 & \dots & 0 \\
               \end{matrix}
             \right)$,
for a unique normalized sequence $S=(S(1),\dots,S(\ell))$.

Indeed, it follows from Step 3 and 4 that $E_0=E(S)$ for some sequence
as in \eqref{eq.Condition_S}. It follows from Proposition \ref{prop.normalized_h} that there is a unique normalized sequence $S=(S(1),\dots,S(\ell))$ and an invertible matrix
$P=P_1\oplus\cdots\oplus P_{\ell+1}\in  GL(d)$, with
$P_i$ a polynomial in $J^{d_i}(0)$ (and thus commuting with $D$),
such that
$PE_0P^{-1}=E(S)$.
\end{proof}

Combining Propositions \ref{prop.normal2} and \ref{prop.ensu} we
obtain the classification of all uniserial $\g_{n,\lambda}$-modules
for $\lambda\ne0$ over an algebraically closed field of characteristic zero.

\begin{theorem}\label{thm.main1} Let $\lambda\neq 0$. Every finite dimensional uniserial representation $T:\g_{n,\la}\to\gl(U)$ satisfying
$\ker(T)\cap V= 0$ is isomorphic to one and
only one normalized representation  $R_{\vec d,\alpha,S}$
with $\vec d=(d_1,\dots,d_{\ell+1})$
satisfying $|\vec d| =\dim U$ and
\[
 \max\{d_i+d_{i+1}:1\le i\le \ell\}=n+1.
\]
\end{theorem}

\begin{remark}
The integer $\ell$ is determined by $T$, more precisely by
$\ker(T)$ as it will be shown in Theorem \ref{thm.main3}.
\end{remark}

\subsection{Relatively faithful representations of the Lie algebra
\texorpdfstring{$\g_{n,\lambda,\ell}$}{gnl,l}}\label{subsec.rel_faithful}

Recall that given an integer $\ell \ge 1$
and a vector space $V$ of dimension $n\geq 2$,
$\N_\ell(V)$ is the free $\ell$-step nilpotent Lie algebra associated to $V$,
and that
$\g_{n,\lambda,\ell}=\langle x\rangle\ltimes \N_\ell(V)$ is the semidirect product where $x$ acts on $V$ via a single Jordan block $J_n(\la)$
(and its action is extended to $\N_\ell(V)$ so that $x$ becomes a Lie algebra derivation).
In this section we will use
Theorem \ref{thm.main2}
to give a complete classification of all uniserial representations of the solvable Lie algebra $\g_{n,\lambda,\ell}$ for $\lambda\ne 0$.

\begin{definition}
We say that a representation $T:\g_{n,\lambda,\ell}\to\gl(U)$ is
\emph{relatively faithful} if $\ker(T)\cap V= 0$ and
$\N_\ell(V)^{\ell-1}\not\subset\ker(T)$ (that is, there is $Z$ in the center of
$\N_\ell(V)$ such that $T(Z)\ne0$).
\end{definition}

In order to classify all uniserial representations of $\g_{n,\lambda,\ell}$
it suffices to classify those that are relatively faithful.
Indeed, let $T:\g_{n,\lambda,\ell}\to\gl(U)$ be a uniserial representation. If $V\subseteq \ker(T)$
then $T$ is determined by a uniserial representation $\langle x\rangle\to\gl(U)$.
The Jordan normal form suffices to classify such representations. We may thus assume
without loss of generality that $V$ is not contained in $\ker(T)$.
Now, if $0\neq \ker(T)\cap V\neq V$, then $T$ is determined by a uniserial representation ${T'}:\g_{n',\lambda,\ell}\to\gl(U)$, where
$\g_{n',\lambda,\ell}
=\langle x\rangle\ltimes \N_\ell(V')$, $V'$ is a factor of $V$ by an $x$-invariant subspace, $x$ acts on $V'$ via an invertible Jordan block $J_{n'}(\la)$, $1\leq n'<n$, and $\ker(T')\cap V'= 0$. Hence, we may assume without loss of generality that $\ker(T)\cap V= 0$.
Let $1\le \ell_0 \le \ell$ be the smallest positive integer such that $\N_\ell(V)^{\ell_0-1}$ is not contained in $\ker(T)$.
Then $T$ is determined by a relatively faithful uniserial
representation $T':\g_{n,\lambda,\ell_0}\to\gl(U)$.

In order to state the classification we introduce the following notation.
\begin{notation}
Given sequences
$\vec d=(d_1,\dots,d_k)$ and $S=(S(1),\dots, S(k-1))$,
we sat that the pair $\vec d$, $S$
is of \emph{extreme type} if
\begin{enumerate}
    \item $\vec d$ is odd-symmetric with $d_1=d_k=1$ and
    \item $S$ is $\phi$-invariant.
\end{enumerate}
\end{notation}

\begin{theorem}\label{thm.main3} Let $n>1$, $\ell\ge 1$ and suppose $\lambda\ne0$.
Then any finite dimensional relatively faithful uniserial
representation $T:\g_{n,\la,\ell}\to\gl(U)$
is isomorphic to one and only one normalized representation
$\bar R_{\vec d,\alpha,S}$
with $\vec d=(d_1,\dots,d_k)$ and $S=(S(1),\dots, S(k-1))$ satisfying:
\begin{enumerate}
    \item $|\vec d| =\dim U$,
    \item  $\max\{d_i+d_{i+1}:1\le i <k\}=n+1$,
     \item if $\ell$ is even then $k=\ell+1$ and the pair  $\vec d,S$ is not of extreme type,
     \item if $\ell$ is odd then either $k=\ell+1$ or we have:
     $k=\ell+2$ and the pair  $\vec d,S$ is of extreme type.
\end{enumerate}
\end{theorem}

\begin{proof} Let $T:\g_{n,\la,\ell}\to\gl(U)$ be a relatively faithful uniserial representation of dimension $d$.
We may extend $T$ obtaining a uniserial representation
$T':\g_{n,\la}\to\gl(U)$ of
$\g_{n,\la}$ such that $\ker(T)\cap V\ne 0$.
It follows from Theorem \ref{thm.main1} that $T'$
is isomorphic, as a representation of $\g_{n,\la}$,
to one and
only one normalized representation  $R_{\vec d,\alpha,S}$
with $\vec d$ satisfying $|\vec d| =\dim U$ and $\max\{d_i+d_{i+1}:1\le i <k\}=n+1$.
Therefore $T$, the original representation of $\g_{n,\la,\ell}$,
is isomorphic to  $\bar R_{\vec d,\alpha,S}$.
Uniqueness follows from Proposition \ref{prop.normal2}.

Since $T'$ is obtained from $T$, we must have  $\L(V)^{\ell}\subset\ker(T')$ and
since $T$ is  relatively faithful, we have
$\L(V)^{\ell-1}\not\subset\ker(T')$.
Therefore, the nilpotency degree of
$\h(\alpha, \lambda,S)$, the image of $T$, is $\ell$.
Now (3) and (4) follow from Theorem \ref{thm.main2}.
\end{proof}

\section{Questions}
We point out that the classification of all 
uniserial representations of the \emph{nilpotent} Lie algebra 
$\g_{n,\la,\ell}$  remains incomplete for $\lambda=0$. 
Also, we know very little about the uniserials of $\g_{n,\la,\ell}$  in positive characteristic, see for instance Question \ref{ques1}. 

 Let $\g$, $\n$ and $\p$ be as in Proposition \ref{prop.ensu}, and suppose that $\g$ is generated by $x$ and $\p$ as Lie algebra.
When will (\ref{s1}) and (\ref{s2}) extend to a representation of $\g$? This extension is clearly unique, if it exists.
We are mostly interested in the case when $\n$ is nilpotent and $\n^q= 0$. Since $\n$ need not be free $\ell$-step nilpotent, the
answer will depend on the actual sequences $\vec d$ and $S$, so that the definition relations of $\n$
will be satisfied. Some positive and negative cases can be found using the Heisenberg Lie algebra $\n=\h(2m+1)$
and in \cite{CGS2}. There may be some $\n$ so that no extension exists for any choice of  $\vec d$ and $S$.


\end{document}